\pdfoutput=1
\RequirePackage{ifpdf}
\ifpdf % We~are running pdfTeX in pdf mode
\documentclass[pdftex]{sigma}
\else
\documentclass{sigma}
\fi

\numberwithin{equation}{section}

\newtheorem{Theorem}{Theorem}[section]
\newtheorem{Corollary}[Theorem]{Corollary}
\newtheorem{Lemma}[Theorem]{Lemma}
\newtheorem{Proposition}[Theorem]{Proposition}
 { \theoremstyle{definition}
\newtheorem{Definition}[Theorem]{Definition}
\newtheorem{Example}[Theorem]{Example}
\newtheorem{Remark}[Theorem]{Remark} }

\newcommand{\C}{{\mathbb{C}}}
\newcommand{\Z}{{\mathbb{Z}}}
\newcommand{\<}{{\langle}}
\renewcommand{\>}{{\rangle}}
\newcommand{\CL}{{\mathcal{L}}}
\renewcommand{\ker}{{\rm{ker}}}
\newcommand{\tens}{\otimes}
\newcommand{\id}{{\rm id}}
\renewcommand{\o}{{}_{\scriptstyle(1)}}
\renewcommand{\t}{{}_{\scriptstyle(2)}}
\renewcommand{\th}{{}_{\scriptstyle(3)}}
\newcommand{\fo}{{}_{\scriptstyle(4)}}
\newcommand{\extd}{{\rm d}}
\newcommand{\del}{{\partial}}
\newcommand{\eps}{\epsilon}
\newcommand{\co}{{}_{\scriptstyle \bar{(1)}}}
\newcommand{\rz}{{}_{\scriptstyle \bar{(0)}}}
\newcommand{\la}{{\triangleright}}
\newcommand{\ra}{{\triangleleft}}
\newcommand{\cg}{{\mathfrak{g}}}
\newcommand{\tilder}{\grave}
\newcommand{\hatr}{\breve }
\newcommand{\hatl}{\check }
\newcommand{\ver}{{\rm ver}}

\begin{document}
%\allowdisplaybreaks

\newcommand{\arXivNumber}{1903.12006}

\renewcommand{\thefootnote}{}

\renewcommand{\PaperNumber}{006}

\FirstPageHeading

\ShortArticleName{Poisson Principal Bundles}
\ArticleName{Poisson Principal Bundles\footnote{This paper is a~contribution to the Special Issue on Noncommutative Manifolds and their Symmetries in honour of~Giovanni Landi. The full collection is available at \href{https://www.emis.de/journals/SIGMA/Landi.html}{https://www.emis.de/journals/SIGMA/Landi.html}}}

\Author{Shahn MAJID and Liam WILLIAMS}
\AuthorNameForHeading{S.~Majid and L.~Williams}
\Address{School of Mathematical Sciences, Queen Mary University of London,\\ Mile End Rd, London E1 4NS, UK}
\Email{\href{mailto:s.majid@qmul.ac.uk}{s.majid@qmul.ac.uk}, \href{mailto:l.williams@qmul.ac.uk}{l.williams@qmul.ac.uk}}

\ArticleDates{Received June 11, 2020, in final form January 05, 2021; Published online January 13, 2021}

\Abstract{We semiclassicalise the theory of quantum group principal bundles to the level of Poisson geometry. The total space $X$ is a Poisson manifold with Poisson-compatible contravariant connection, the fibre is a Poisson--Lie group in the sense of Drinfeld with bicovariant Poisson-compatible contravariant connection, and the base has an inherited Poisson structure and Poisson-compatible contravariant connection. The latter are known to be the semiclassical data for a quantum differential calculus. The theory is illustrated by the Poisson level of the $q$-Hopf fibration on the standard $q$-sphere. We also construct the Poisson level of the spin connection on a principal bundle.}


\Classification{58B32; 53D17; 17B37; 17B62}

\renewcommand{\thefootnote}{\arabic{footnote}}
\setcounter{footnote}{0}

\vspace{-1mm}

\section{Introduction}
It is known since the work of Drinfeld in the 1980's \cite{Dri:ham,Drinfeld} that the infinitesimal notion of a quantum group is that of a Lie bialgebra. This is a Lie algebra $\mathfrak{g}$ equipped with a `Lie cobracket' map $\delta$ that forms a Lie $1$-cocycle and makes $\mathfrak{g}^*$ into a Lie algebra. The associated connected and simply connected Lie group $G$ is a Poisson--Lie group, the semiclassical analogue of a quantum group; one can think of standard $q$-deformed quantum groups $\C_q[G]$ as in~\cite{FRT}, as the algebraic version of noncommutative deformations of $C^\infty(G)$ (which is essentially recovered in some algebraic form) if we let $\lambda\to 0$, where $q=e^{\lambda\over 2}$. These are dual to the Drinfeld--Jimbo $U_q(\cg)$ enveloping algebra deformations. The same applies to the bicrossproduct family of quantum groups associated to Lie algebra factorisations~\cite{Ma:book} which tend to coordinate algebras of inhomogeneous groups or dually to their enveloping algebras and have a role in quantum spacetime models.

Meanwhile, quantum groups also led in the 1990s to the emergence of a constructive approach to quantum differential geometry including but not limited to their quantum geometry. By now, there is a significant body of work and we refer to the text~\cite{BegMa} and references therein. This approach is somewhat different in character from Connes'~\cite{Con} well-known approach to noncommutative geometry based on generalising the Dirac operator (a `spectral triple') using operator algebras, as well as from noncommutative-algebraic geometric approaches such as~\cite{SvdB}. The constructive approach starts with the notion of 1-forms $\Omega^1$ over a possibly noncommutative algebra~$A$. Sections of vector bundles appear in the nicest case as $A$-$A$-bimodules $E$ and connections on them as bimodule connections $\nabla_E\colon E\to \Omega^1\tens_A E$ in the sense of~\cite{DVM,Mou}. There is also a notion of quantum principal bundles with quantum group fibre, `spin connections' on them and associated bundles~\cite{BrzMa}. The most well-known example is the $q$-Hopf fibration~\cite{BegMa, BrzMa,HajMa,Ma:spin}, with total space algebra $P=\C_q[{\rm SU}_2]$, quantum group fibre $H=\C_{q^2}\big[S^1\big]$ and base algebra $A=\C_q\big[S^2\big]$. We introduce a Poisson-level theory relevant to the first order data for deformation-quantisation of such quantum principal bundles and connections on them. This fits within a programme~\cite{BegMa:sem, BegMa:prg, FriMa,MaTao:plg} to semiclassicalise various aspects of noncommutative differential geometry, including quantum Riemannian geometry.

Our results effectively extend Drinfeld's work to a Poisson--Lie group $G$ as the structure group of a classical principal bundle, the total space $X$ of which is a Poisson manifold; the action of~$G$ is a Poisson action and, which is the critical new part, both~$G$ and $X$ are equipped with the leading order data for quantum differential forms $\Omega^1$ in a compatible way. In this regard, \cite{Bur,Fer,Hue} previously looked at deformations of vector bundles and found their Poisson-level data as a~{\em contravariant} or Lie--Rinehart connection $\hat\nabla$ (we will use a hat to remind us that these are not usual connections). The specific case of the deformation of the bimodule of 1-forms $\Omega^1$ was studied in~\cite{BegMa:sem, BegMa:twi, Haw} and requires $\hat\nabla$ to be Poisson-compatible and flat for a~strictly associative calculus at the next order of deformation. A main result \cite[Corollary~4.2]{MaTao:plg} in the case of a Poisson--Lie group $G$ is that flat $\hat\nabla$ for the semiclassical version of left translation-invariant $\big(\Omega^1,\extd\big)$ correspond to $\cg^*$, the dual of the Lie algebra of $G$, being a pre-Lie or Vinberg algebra. Meanwhile, one of the powerful lessons of quantum geometry model building is that very often no quantum differential calculus of the expected classical dimension exists. There are broadly two remedies to this `quantum anomaly for differentiation'. One is to allow an extra dimension in the cotangent bundle to absorb the anomaly~\cite{Ma:spo,Ma:rec}, which exits a deformation theory setting. The other is to allow a breakdown of associativity at order~$\lambda^2$ in the deformation parameter~$\lambda$, as in~\cite{BegMa:sem, BegMa:twi}. This corresponds at the Poisson level to allowing~$\hat\nabla$ to have curvature, which we do throughout.

Because this work stands at a triple point of Poisson geometry, noncommutative differential geometry and deformation theory, there are necessarily significant preliminaries, collected in Section~\ref{secpre}. This will entail compromises in notation (such as both left invariant and right invariant vector fields, instead of keeping everything to the left). We also include key Lemmas~\ref{PLGprop} and~\ref{PLGactionprop} known to experts on Poisson--Lie groups but which should be helpful to those coming from noncommutative geometry. The first genuinely new result is Proposition~\ref{taocovarianceprop} which better characterises one-sided translation invariance of $\hat\nabla$ on a Poisson--Lie group as
\[ \xi\la\hat\nabla_{\eta}\tau=\hat\nabla_{\xi\la\eta}\tau+\hat\nabla_{\eta}(\xi\la\tau) +i_{\tilde{\delta_\xi^1}}(\eta)\big(\delta_\xi^2\la\tau\big)\]
for all $\xi\in\cg$ and 1-forms $\eta$, $\tau$, with notation $\delta(\xi)=\delta^1_\xi\tens\delta^2_\xi$ (sum of terms understood). Here $i$ is interior product and $\xi\la$ is Lie derivative along the left-invariant vector field $\tilde\xi$ for the infinitesimal action of $G$ on itself. Theorem~\ref{bicovXi} gives a Lie-bialgebra characterisation of 2-sided translation invariant (bicovariant) $\hat\nabla$ in terms of data $\Xi\colon \cg^*\tens\cg^*\to \cg^*$. Section~\ref{PLGact} polarises these concepts to the case of a Poisson Lie group $G$ Poisson-acting on a Poisson manifold $X$ with a $G$-covariant Poisson-compatible contravariant connection $\hat\nabla$ on $X$.

This then feeds into our novel formulation of a Poisson principal bundle in Section~\ref{PPB}. The key new idea here is a transversality condition
\[ i_{\tilde\xi}\hat\nabla_\eta\tau=\pi_X(\eta,\extd i_{\tilde\xi}\tau)+i_{\widetilde{\Xi_\xi^{*1}}}(\eta) i_{\widetilde{\Xi_\xi^{*2}}}(\tau)\]
for all $\xi\in \cg$ and $\eta,\tau \in \Omega^1(X)$, see Definition~\ref{PPBdef}. Here $\pi_X$ is the Poisson bivector and $\Xi^*(\xi)=\Xi_\xi^{*1}\tens \Xi_\xi^{*2}$ (sum of terms understood). An important consequence is that the base $M$ of the bundle inherits a Poisson structure with Poisson-compatible contravariant connection $\hat\nabla^M$. This fits with our point of view of such data being the right notion of a `semiclassical manifold', as infinitesimal data for the quantisation of both the algebra {\em and} its differential structure. Along the way, Corollaries~\ref{corconcovX} and~\ref{corintsch} provide some consequences on the Schouten bracket $[\eta,\tau]$ of 1-forms in this context.

The final Section~\ref{Assocsec} formulates the semiclassical notion of a spin connection on a quantum principal bundle and quantum associated bundles. We first formulate a quantum spin connection as an equivariant bimodule connection $\nabla_P\colon P\to \Omega^1_{\rm hor}$. Then we show that the semiclassical data for the quantisation of a classical such $\nabla_P$ is a map $\Gamma\colon C^\infty(X)\to \Omega^1_{\rm hor}$ such that
\[ \Gamma(ap)=a\Gamma(p)+ \hat\nabla^X_{\extd a}\nabla_P p-\hat\nabla^X_{\extd p}\extd a-\nabla_P\{a,p\}_X\]
for all $a\in C^\infty(M)$, $p\in C^\infty(X)$. The main result, Theorem~\ref{semiclassbundleconn}, is that every classical spin connection has a canonical such $\Gamma$ for the semiclassical level. The proof makes critical use of the bundle transversality condition as well as bicovariance of $\hat\nabla^G$ on the fibre, bringing together our results.

The paper concludes with some directions for further work. Note that the seminal work \cite{Fer} considered a notion of contravariant connections on ordinary principal bundles over a Poisson manifold that induce contravariant connections on associated bundles. This is not what we do (we consider the deformation of usual and not contravariant connections on Poisson principal bundles), but is certainly of interest. Meanwhile, \cite{Bor} showed existence and uniqueness of the deformation-quantisation of principal bundles in sense of $C^\infty(X)[[\lambda]]$ a right module over the $*$-product $C^\infty(M)[[\lambda]]$ equivariantly under the structure group $G$, which is not quantised. The latter could be seen as a special case of a PLG with zero Poisson bracket and even here, the paper found obstructions for the bimodule case. These are very different considerations from the ones in the present paper but they support our view that we may need nonassociative geometry at higher deformation order.

\section{Preliminaries}\label{secpre}

\subsection{Elements of noncommutative differential geometry}\label{Omega1}
We outline a formalism of noncommutative differential geometry in a bimodule approach, including Hopf algebras with differential structure. The classical model is $A=C^\infty(M)$ for a~smooth real manifold~$M$ with exterior algebra $\Omega(M)$, while the general formulation works over any unital algebra with an algebraically defined exterior algebra~$\Omega_A$.

Thus, a `differential calculus' on a unital algebra $A$ consists of a graded algebra $\Omega_A=\oplus\Omega^n_A$ with $n$-forms $\Omega^n_A$ for $n\geq 0$, an associative product $\wedge\colon \Omega^n_A\tens\Omega^m_A\rightarrow\Omega^{n+m}_A$ and an exterior derivative $\extd\colon \Omega^n_A\rightarrow\Omega^{n+1}_A$ satisfying the graded-Leibniz rule. Here $\Omega^0_A=A$ and we assume that~$\Omega_A$ is generated by degree $0$, $1$ and that~$\Omega^1_A$ is spanned by products of~$A$ and~$\extd A$. Both conditions are slightly stronger than a DGA in homological algebra.

Vector bundles are expressed as (typically projective) $A$-modules $E$ (classically, this would be the space of sections). A left connection on this is $\nabla_E\colon E\rightarrow\Omega^1_A\tens_AE$ obeying the left Leibniz rule $\nabla_E(ae)=\extd a\tens e+a\nabla_Ee$ for all $a\in A$, $e\in E$. We have a {\em bimodule connection}~\cite{DVM,Mou} if $E$ is a bimodule and there is a bimodule map with
\[ \sigma_E\colon \ E\tens_A\Omega^1_A\rightarrow\Omega^1_A\tens_AE, \qquad \nabla_E(ea)=(\nabla_Ee)a+\sigma_E(e\tens_A\extd a).\]
If $\sigma_E$ is well-defined then it is uniquely determined, so its existence is a property of a left connection on a bimodule. Bimodule connections extend to tensor products, giving us a monoidal category, where the tensor product connection on $E\tens_A F$ uses $\sigma_E$ to reorder the output of $\nabla_F$. More details and results are in~\cite{BegMa}.

The notion of a Hopf algebra is covered in several texts, e.g., \cite{Ma:book}. Briefly, this means a~unital algebra $H$ which is also a coalgebra with counit $\eps\colon H\to k$ (for $k$ the ground field) and coproduct $\Delta\colon H\to H\tens H$, with these maps being algebra homomorphisms, and for which an antipode $S\colon H\to H$ exists. The latter is required to obey $(Sh\o)h\t=\eps(h)=h\o Sh\t$ in the `Sweedler notation' $\Delta h=h\o\tens h\t$ (sum of such terms implicit). A left action of a~classical group on a vector space $V$ corresponds in our function algebra point of view to a right coaction $\Delta_R\colon V\to V\tens H$ obeying the arrow-reversal of the usual axioms for a right action. If $V=A$ then we say that $A$ is a {\em comodule algebra} if $\Delta_R$ is an algebra homomorphism. If we have a~suitable dually paired Hopf algebra (the algebraic dual typically being to big) then $A$ becomes a {\em module algebra} for the left action $\la$ given by evaluation against the left factor of the output of~$\Delta_A$, meaning $\phi\la(ab)=(\phi\o\la a)(\phi\t\la b)$ and $\phi\la 1=\eps(\phi)1$ for $\phi$ in the dually paired Hopf algebra. If $H$ coacts from the left then we similarly have a right action $\ra$ by evaluation against the left factor of the output of~$\Delta_L$.

In particular, a Hopf algebra coacts on itself as a comodule algebra from both the left and the right via $\Delta$. A calculus $\Omega^1_H$ is left (resp.\ right) translation covariant if the coaction extends to $\Omega^1$ by
\[\Delta_L(h\extd g)=h\o g\o\tens h\t \extd g\t,\qquad \Delta_R(h\extd g)=h\o \extd g\o\tens h\t g\t\]
(making it a `Hopf module') with the coaction commuting with $\extd$. A calculus is called bicovariant if both coactions extend, and strongly bicovariant if $\Delta_L+\Delta_R$ makes the exterior algebra $\Omega_H$ into a super-Hopf algebra. We will not need higher forms, but this is true for the canonical $\Omega_H$ in the bicovariant case in \cite{Wor}.

Finally, and equally briefly, the data for a quantum principal bundle is $P$ a right $H$-comodule algebra with `base' the algebra $A=P^H=\{p\in P\, |\, \Delta_Rp=p\tens 1\}$ of elements fixed under $\Delta_Rp=p\rz \tens p\co $ in a compact notation. We assume that $\Omega^1_P$ is $H$-covariant in that $\Delta_R$ extends in a similar way to right-covariance of~$\Omega^1_H$ above. In this case $\Omega^1_A=A\extd A$ computed in $\Omega^1_P$ is the inherited calculus. We also assume that $\Omega^1_H$ is bicovariant and let $\Lambda^1_H$ be the space of left-invariant 1-forms. Next, we assume that there is a well-defined map ${\rm ver}\colon \Omega^1_P\to P\tens \Lambda^1_H$ (the generator of vector fields for the infinitesimal action) given by
\begin{equation}\label{PHver} \text{ver}(p\extd q)=pq\rz \tens\big(Sq\co\o \big)\extd q_{\bar{(1)}((2)}=pq\rz\tens\varpi \pi_\eps\big(q\co\big), \qquad \forall\, p,q\in P, \end{equation}
where $\pi_\eps\colon H\to H^+$ defined by $\pi_\eps(h)=h-1\eps(h)$ projects onto the kernel of the counit, and the Maurer--Cartan form $\varpi \colon H^+\to \Omega^1$ is defined by $\varpi (h)=Sh\o\extd h\t$ and affords an identification of $\Lambda^1_H$ as a quotient of $H^+$. Finally, as a replacement for local triviality, we require a short exact sequence of left $P$-modules
\begin{equation}\label{PHexact} 0\longrightarrow P\Omega^1_AP\longrightarrow \Omega^1_P\xrightarrow{\text{ver}}P\tens\Lambda^1_H\longrightarrow 0. \end{equation}
More details are in \cite{BrzMa, Ma:spin} and \cite[Chapter~5]{BegMa}. We will at certain points need to discuss this algebraic theory, adapted to a putative formal deformation version, but only to the point of extracting Poisson level data which then make sense by themselves in the smooth setting.

\subsection{Elements of deformation theory}\label{secpredef}

In deformation theory, it is usual to adapt the algebraic context above by now working over the ring $\C[[\lambda]]$ for a formal deformation parameter $\lambda$ and with an associative $\C[[\lambda]]$-linear $*$-product of the form $a*b=\sum_{i=0}^\infty C_i(a,b)\lambda^i$ on $A=C^\infty(M)[[\lambda]]$, see, e.g.,~\cite{Bur}. In this context, it is well known that
\[ a* b-b * a=\lambda\lbrace a,b\rbrace+\mathcal{O}\big(\lambda^2\big). \]
defines a Poisson bracket $\{a,b\}=C_1(a,b)-C_1(b,a)$ for $a,b\in C^\infty(M)$. We denote the associated Poisson bivector by $\pi$ so that $\lbrace a,b\rbrace=\pi(\extd a,\extd b)$. Similarly, if we have a vector bundle over~$M$ with space of sections $\Gamma$ as a projective module over $C^\infty(M)$, and this is deformed to an $A$-bimodule $\Gamma[[\lambda]]$ with left and right bullet products $a\bullet \eta=\sum_{i=0}^\infty L_i(a,\eta)\lambda^i$ and $\eta\bullet a=\sum_{i=0}^\infty R_i(\eta,a)\lambda^i$ then
\[a\bullet\eta-\eta\bullet a=\lambda\hat\nabla_{\extd a}\eta+\mathcal{O}\big(\lambda^2\big) \]
for all $a\in C^\infty(M), \eta\in \Gamma$, where $\hat\nabla_{\extd a}\eta:=L_1(a,\eta)-R_1(\eta,a)$ is a {\em contravariant connection} or `Lie--Rinehart' covariant derivative~\cite{Bur,Fer, Hue} characterised by
\begin{equation}\label{precon1} \hat\nabla_{a\zeta}=a\hat\nabla_\zeta,\qquad \hat\nabla_\zeta(a\eta)=\pi(\zeta,\extd a)\eta+ a\hat\nabla_\zeta\eta,\end{equation}
for all $\zeta\in \Omega^1(M)$. In the strictly associative case, this will be flat (as a special case of a more general statement in~\cite{Bur}). We will be interested in the case $\Gamma=\Omega^1(M)$, the classical space of 1-forms, and for $\Omega^1[[\lambda]]$ to obey the Leibniz rule with an undeformed exterior derivative $\extd$, which requires the further {\em Poisson-compatibility}
\begin{equation}\label{precon2} [\eta,\tau]=\hat{\nabla}_\eta\tau-\hat{\nabla}_\tau\eta\end{equation}
with respect to a `Schouten bracket' $[\eta,\tau]$ of 1-forms defined by $[\extd a,\extd b]=\extd\{a,b\}$ along with $[a\eta,b\tau]=ab[\eta,\tau]+a\pi(\eta,\extd b)\tau-b\pi(\tau,\extd a)\eta$. More usually in the literature (see, e.g., \cite{Fer}), condition (\ref{precon2}) is called `torsionless' and Poisson-compatible refers to something stronger, whereas we reserve torsion for its usual meaning for linear connections. Equations (\ref{precon1})--(\ref{precon2}) are the geometric version on general 1-forms of the conditions \cite{BegMa:sem,BegMa:prg, Haw}
\begin{gather*} \hat\nabla_{\extd(ab)}\eta=a\hat\nabla_{\extd b}\eta+ (\hat\nabla_{\extd a}\eta) b,\qquad \hat\nabla_{\extd a}(b\eta)=b\hat\nabla_{\extd a}(\eta)+\lbrace a,b\rbrace\eta, \\
\extd\lbrace a,b\rbrace=\hat\nabla_{\extd a}\extd b-\hat\nabla_{\extd b}\extd a\end{gather*}
coming out of the order $\lambda$ analysis. The point of view in \cite{BegMa:sem,BegMa:prg} was to think if this equivalently as a partially defined connection or preconnection $\nabla$ by $\hat\nabla_{\extd a}=\nabla_{\hat a}$, where $\hat{a}=\lbrace a,\;\rbrace$ is the Hamitonian vector field associated to $a\in C^\infty(M)$ and $\nabla$ is a usual connection but only defined along such vector fields. Both points of view are useful. In the symplectic case, the data is equivalent to a linear connection which preserves the Poisson tensor up to torsion, and the contravariant $\hat \nabla$ is its pull back along the map $\pi^\#\colon \Omega^1(M)\to\rm{Vect}(M)$ given by $\pi^\#(\eta)=\pi(\eta,\ )$ with $\pi^\#(\extd a)=\hat a$.

 The strictly associative setting is, however, often not compatible with deformation theory in that for examples related to quantum groups, one may not have $\Omega^1$, $\extd$ with desired covariance properties and with expected classical dimension over the noncommutative algebra (i.e., even in the free module case where there is a dimension). This was noted at the Poisson level in \cite{BegMa:sem, Haw} and leads into a somewhat different point of view~\cite{BegMa:prg} where we consider formal products of the form
\begin{gather} a*b=ab+{\lambda\over 2}\{a,b\}+ \mathcal{O}\big(\lambda^2\big),\qquad a\bullet \eta=a\eta+{\lambda\over 2}\hat\nabla_{\extd a}\eta+\mathcal{O}\big(\lambda^2\big),\nonumber\\
 \eta\bullet a=a\eta-{\lambda\over 2}\hat\nabla_{\extd a}\eta+\mathcal{O}\big(\lambda^2\big)\label{proddef}\end{gather}
but demand associativity and the bimodule property only $\mathcal{O}\big(\lambda^2\big)$. The minimal setting is to assume only that $\{\ ,\ \}$ is a biderivation and $\hat\nabla$ a possibly curved contravariant connection compatible with it. This leads into fully nonassociative geometry, while the most important special case appears to be the intermediate one where~$*$ is associative (deformation-quantising a Poisson manifold) but~$\bullet$ is not necessarily a bimodule structure $\mathcal{O}\big(\lambda^2\big)$, so that $\{\ ,\ \}$ is a Poisson bracket but $\hat\nabla$ can be curved. This leads to a self-contained setting, which we adopt, of a Poisson manifold equipped with possibly curved Poisson-compatible contravariant connection.

\subsection{Elements of Poisson--Lie theory}\label{PLGprelim}

Following the work of Drinfeld \cite{Drinfeld}, the semiclassical object underlying a deformation Hopf algebra is a {\em Poisson--Lie group} (PLG). This means a Lie group $G$ which is also a~Poisson manifold such that the multiplication map $G\times G\rightarrow G$ is a Poisson map, where $G\times G$ is equipped with the product Poisson structure. In terms of the Poisson tensor $\pi$, this is
\[ \pi(gh)=L_{g*}(\pi(h))+R_{h*}(\pi(g))\]
for all $g,h\in G$. Here $R_g\colon G\to G$ and $L_g\colon G\to G$ are right and left translation by $g$ and have derivatives $R_{h*}\colon T_gG\to T_{gh}G$ and $L_{g*}\colon T_hG\to T_{gh}G$. The corresponding data at the Lie algebra level is a {\em Lie bialgebra}, i.e., $(\mathfrak{g},\delta)$ where $\mathfrak{g}$ is a Lie algebra and $\delta\colon \mathfrak{g}\rightarrow\mathfrak{g}\wedge\mathfrak{g}$ is a Lie coalgebra (so its dual is a Lie bracket map $\mathfrak{g}^*\tens\mathfrak{g}^*\rightarrow\mathfrak{g}^*$ on $\mathfrak{g}^*$), with $\delta$ a $1$-cocycle on $\mathfrak{g}$ relative to the adjoint representation of $\mathfrak{g}$ on $\mathfrak{g}\wedge\mathfrak{g}$ in the sense
\[\delta([\xi,\eta])=\rm{ad}_{\xi}(\delta(\eta))-\rm{ad}_{\eta}(\delta(\xi)),\qquad \forall\, \xi,\eta\in\mathfrak{g}. \]

Next, the right and left actions of the group on itself define (respectively) left and right actions on the algebra of functions on the group $C^\infty(G)$. We define the {\em left-translation invariant formulation} as follows. There is a {\em right} action
\[ (a\ra h)(g)=a(hg), \qquad\forall\, g,h\in G, a\in C^\infty(G) \]
(usually made into a left action via the group inverse, but we refrain from this). Setting $g=e^{t\xi}$ and differentiating at $t=0$ gives a right action of the Lie algebra $\cg$,
\[ (a\ra\xi)(g)=\frac{\extd}{\extd t}\Big|_0a\big(e^{t\xi}g\big)=\tilder\xi(a)(g), \]
defining an associated {\em right-invariant vector field} $\tilder\xi$. This can be defined also as $\tilder\xi(g)=R_{g*}(\xi)$. The actions extends to 1-forms by
\[ (\tau\ra h)(g)=L_h^*(\tau(hg))\]
and $\ra\xi=\CL_{\tilder\xi}=i_{\tilder\xi}\extd+\extd i_{\tilder\xi}$ on forms, where $i$ denotes interior product and $\CL$ the Lie derivative. We similarly define the {\em left-invariant vector field} $\tilde\xi$ associated to $\xi\in\cg$ as $\tilde\xi(g)=L_{g*}\xi$ for $\xi\in \cg$ equating to $\xi\la$ defined below by right translation. We also recall that there is a one to one correspondence between $1$-forms $\tau\in \Omega^1(G)$ and $\tilde{\tau}\in C^\infty(G,\mathfrak{g}^*)$ via $\tilde{\tau}(g)=L^*_{g}(\tau(g))$ or conversely $s\in C^\infty(G,\mathfrak{g}^*)$ defines a $1$-form which we will denote $\hatl{s}(g)=L^*_{g^{-1}}(s(g))$. In these terms, the action is $(s\ra h)(g)=s(hg)$ as for a function. In particular, $\hatl v(g)=L_{g^{-1}}^*v$ is the left-translation invariant 1-form (invariant under $\ra$) associated to $v\in\cg^*$ and $\<\tilde\xi,\hatl v\>=\<\xi,v\>$ is constant. We also have $\big\langle \xi,\big(\tilde\extd a\big)(g)\big\rangle=\tilde\xi(a)(g)$. Finally, in the right translation-invariant formulation, we start with the left action
\[ (h\la a)(g)=a(gh),\qquad (\xi\la a)(g)={\extd \over\extd t}\Big|_{t=0}a\big(ge^{t\xi}\big)=\tilde{\xi}(a)(g)\]
is the left-invariant vector field $\tilde\xi$ for the infinitesimal action of $\xi\in\cg$. The group action extends to differential forms $\tau\in \Omega^1(G)$ by
\[ (h\la\tau)(g)=R^*_{h}(\tau(gh))\]
with Lie algebra action $\la$ given by Lie derivative along $\tilde\xi$. In this case, the relevant one to one correspondence between $1$-forms $\tau\in \Omega^1(G)$ and, say, $\tilder{\tau}\in C^\infty(G,\mathfrak{g}^*)$ is via $\tilder{\tau}(g)=R^*_{g}(\tau(g))$ with inverse $s\in C^\infty(G,\mathfrak{g}^*)$ defining a $1$-form $\hatr{s}(g)=R^*_{g^{-1}}(s(g))$. In these terms, the action is $(h\la s)(g)=s(gh)$ as for a function. In particular, $\hatr v(g)=R_{g^{-1}}^*v$ is a right-translation invariant 1-form (invariant under $\la$) associated to $v\in\cg^*$. We have $\<\tilder\xi,\hatr v\>=\<\xi,v\>$ and $\<\xi,(\tilder\extd a)(g)\>=\tilder\xi(a)(g)$ where $\tilder\xi(g)=R_{g*}\xi$ is the right-invariant vector field associated to $\xi\in\cg$.

Using this second formulation, we can view any Poisson structure on a Lie group as the right translation of some map $D\colon G\to\mathfrak{g}\tens\mathfrak{g}$, i.e., $\pi(g)=R_{g*}(D(g))$. We define
\begin{equation*}%\label{deltaD}
\delta\xi:= {\extd\over\extd t}\Big|_{t=0} D\big(e^{t\xi}\big)\in \cg\tens\cg,\qquad\forall\, \xi\in\mathfrak{g}.\end{equation*}
Drinfeld showed that if $(\cg,\delta)$ is a Lie bialgebra then the associated connected and simply connected Lie group $G$ is a Poisson--Lie group by exponentiating $\delta$ from a Lie algebra to a Lie group cocycle. Conversely, if $(G,D)$ is a Poisson--Lie group, then its Lie algebra is a Lie bialgebra by differentiating $D$ at the identity. Here $\delta$ on a Lie bialgebra is a~$1$-cocycle hence exponentiates to a 1-cocycle $D\in Z^1_{{\rm Ad}}(G,\mathfrak{g}\tens\mathfrak{g})$. More details are in~\cite{Ma:book}.

Drinfeld~\cite{Drinfeld} also introduced a deformation theory point of view where quantum groups $U_q(\cg)$ are deformations of $U(\cg)$ as a Hopf algebra with $q=e^{\lambda\over 2}$. Suffice it to say that the coproduct homomorphism property applied to $\Delta(\xi\eta-\eta\xi)$ leads at order $\lambda$ to the cocycle axiom for $\delta:=\Delta-\tau\circ\Delta$, with $\tau$ the transposition map. Meanwhile, if given $\delta$, we can suppose a specific form of deformation,
\begin{equation}\label{Deltadef} \Delta \xi=\xi\tens 1+ 1\tens\xi+\frac{\lambda}{2}\delta(\xi) + \mathcal{O}\big(\lambda^2\big)\end{equation}
for $\xi\in\cg$ as the analogue on the coalgebra side of the specific form in~(\ref{proddef}) on the algebra side. Here $\delta$ can be viewed as a Lie bracket on $\cg^*$ which in turn is the Poisson bracket of linear functions on~$G$ near the identity. These quantum groups are dual to corresponding $\C_q[G]$ which again have formal deformation versions. A complication here is that $C^\infty(G)$ is only a Hopf algebra with respect to a topological tensor product, but this can be handled in the compact connected case~\cite{Bon}. In practice, one can work with a dense algebraic model $\C[G]$, generated in the case of classical Lie type by the matrix entry coordinate functions $t^i{}_j$ according to $\rho\colon G\subset M_n$. Here $\rho$ extends to $U(\cg)$ and provides a dual pairing $U(\cg)$ by $\<h,t^i{}_j\>=\rho(h)^i{}_j$ of Hopf algebras. With such complications understood, we will similarly suppose for the sake of discussion some formal deformation $C^\infty(G)[[\lambda]]$ for the specific leading form of product~(\ref{proddef}) and a dual $U(\cg)[[\lambda]]$ with the leading form~(\ref{Deltadef}) of coproduct.

\section{Poisson--Lie groups and their semiclassical calculi revisited}\label{PLGdeform}

We start with a more geometric characterisation of the PLG condition.

\begin{Lemma}[cf.~\cite{DS,KS,Vai}] \label{PLGprop}Let $G$ be a connected Lie group with Poisson bracket $\{\ ,\ \}$ vanishing at the identity and associated map $D\colon G\to \cg\tens\cg$ with differential at the identity $\delta(\xi)=\delta_\xi^1\tens\delta_\xi^2$ $($sum of such terms understood$)$. Then
\begin{equation}\label{eqPLGprop}
\xi\triangleright\lbrace a,b\rbrace=\lbrace \xi\triangleright a,b\rbrace+\lbrace a,\xi\triangleright b\rbrace+\big(\delta_\xi^1\triangleright a\big)\big(\delta_\xi^2\triangleright b\big),
\end{equation}
for all $\xi\in\mathfrak{g}$, $a,b\in C^\infty(G)$ if and only if $G$ is a PLG via $\{\ ,\ \}$.
\end{Lemma}
\begin{proof} The proof in one direction, where the PLG condition holds, is clear from \cite{KS} but we include it for completeness. This can also be found in~\cite{DS} as well as easily derived from \cite[Proposition~10.7]{Vai}. We recall that $\lbrace a,b\rbrace(g)=\langle\pi(g),\extd a\tens\extd b\rangle$, where evaluation here is the standard one between sections of the tangent bundle and cotangent bundle extended to tensor products. Then
\begin{align*}
\xi\la\{a,b\}&=\tilde\xi(\{a,b\})=\<\CL_{\tilde\xi}\pi,\extd a\tens\extd b\>+\big\langle \pi,\extd\tilde\xi(a)\tens\extd b+\extd a\tens\extd\tilde\xi (b)\big\rangle
\end{align*}
with the last two terms $\{\xi\la a,b\}+\{a,\xi\la b\}$. For the first term, if the PLG condition holds then
\begin{align*} (\CL_{\tilde\xi}\pi)(g)&=(\xi\la\pi)(g)={\extd \over\extd t}\Big|_0 R_{e^{-t\xi}*}\pi\big(ge^{t\xi}\big)={\extd \over\extd t}\Big|_0\big( R_{e^{-t\xi}*}L_{g*}\pi\big(e^{t\xi}\big)+\pi(g) \big)\\
&={\extd \over\extd t}\Big|_0 L_{g*}D\big(e^{t\xi}\big)=L_{g*}\delta\xi,\end{align*}
which evaluates on $\extd a\tens\extd b$ to $\tilde\delta^1_\xi(a)\tilde\delta^2_\xi(b)$, cf.~\cite[p.~54]{KS}. Conversely, if (\ref{eqPLGprop}) holds then we deduce that the differential at $t=0$ of $R_{e^{-t\xi}*}\pi\big(ge^{t\xi}\big)-R_{e^{-t\xi}*}L_{g*}\pi\big(e^{t\xi}\big)$ vanishes, i.e., that ${\rm Ad}_{g*}\delta\xi={\extd \over\extd t}\big|_0 D\big(g e^{t\xi}\big)=\tilde\xi(D)(g)$.
If we suppose that $D(e)=0$ and that $G$ is connected, then~$D$ is a 1-cocycle on the group (by the same argument as in the proof of \cite[Theorem~8.4.1]{Ma:book}). Hence~$G$ is a PLG.
\end{proof}

Next, we consider the semiclassical data for differential calculus, i.e., a Poisson-compatible contravariant connection~$\hat\nabla$. For any contravariant connection, we define
\begin{equation*} \tilde{\nabla}_{\extd a}\colon \ C^\infty(G,\mathfrak{g}^*)\rightarrow C^\infty(G,\mathfrak{g}^*),\qquad \tilde{\nabla}_{\extd a}s:=\tilde{\ }\circ{\hat\nabla}_{\extd a}\hatl s,\end{equation*}
where $\hatl s$ is the 1-form corresponding to $s\in C^\infty(G,\cg^*)$ and $\tilde{\ }$ in this context turns a 1-form back into an element of $C^\infty(G,\cg^*)$, as explained in Section~\ref{PLGprelim}.
The starting point in~\cite{BegMa:sem} is that $\tilde\nabla$ necessarily has the form
\begin{equation}\label{mapXi}
(\tilde{\nabla}_{\extd a}s)(g)=\lbrace a,s\rbrace(g)+\tilde{\Xi}(g,\tilde{\extd a}(g),s(g)),
\end{equation}
for some map $\tilde\Xi\colon G\times\mathfrak{g}^*\times\mathfrak{g}^*\rightarrow\mathfrak{g}^*$ and is left translation covariant (in a manner corresponding to a calculus being left covariant in a Hopf algebra sense) if and only if $\tilde{\Xi}(g,\phi,\psi)$ is independent of $g$ for all $\phi,\psi\in\mathfrak{g}^*$, i.e., given by $\Xi\colon \cg^*\tens\cg^*\to \cg^*$. In this case, the (contravariant) connection is Poisson-compatible if and only if
\begin{equation}\label{Xipoisson} \Xi(\phi,\psi)-\Xi(\psi,\phi)=[\phi,\psi]_{\mathfrak{g}^*}, \qquad \forall\, \phi,\psi\in\mathfrak{g}^*.
\end{equation}
This uses the left-translation invariant formulation associated with right action~$\ra$. We equally well have a right-translation invariant formulation with left action $\la$, and in fact we will focus proofs on this case. In this case, $\tilder{\nabla}_{\extd a}\colon C^\infty(G,\mathfrak{g}^*)\rightarrow C^\infty(G,\mathfrak{g}^*)$ defined by $\tilder{\nabla}_{\extd a}s:=\tilder{\ }\circ{\hat\nabla}_{\extd a}\hatr s$ has the form
\begin{equation*}%\label{rightmapXi}
(\tilder{\nabla}_{\extd a}s)(g)=\lbrace a,s\rbrace(g)+\tilder{\Xi}(g,\tilder{\extd a}(g),s(g))
\end{equation*}
and is right translation covariant if and only if $\tilder\Xi$ is again given by a constant $\Xi\colon \cg^*\tens\cg^*\to \cg^*$, with Poisson-compatibility if and only if~(\ref{Xipoisson}) holds. The latter and the contravariant connection being flat is equivalent to $\Xi$ being a pre-Lie structure for $\cg^*$~\cite{MaTao:plg} (but we do not limit ourselves to this case.) Our next goal is a~more geometric reformulation of these results as following from Lemma~\ref{PLGprop}.

\begin{Proposition}\label{taocovarianceprop} Let $G$ be a connected PLG and $\hat\nabla$ a contravariant connection on $\Omega^1(G)$. Then $\hat\nabla$ is right-translation covariant in the sense
\begin{equation}\label{eqtaocov}\xi\la\hat\nabla_{\eta}\tau=\hat\nabla_{\xi\la\eta}\tau+\hat\nabla_{\eta}(\xi\la\tau) +i_{\tilde{\delta_\xi^1}}(\eta)\big(\delta_\xi^2\la\tau\big)
\end{equation}
for all $\tau,\eta\in\Omega^1(G)$ and $\xi\in\cg$ if and only if $\tilder\Xi$ is constant on~$G$.
\end{Proposition}
\begin{proof} In terms of $s\in C^\infty(G,\cg^*)$ in the right-translation invariant formulation,
\begin{align*} \big(\xi\la\tilder\nabla_{\extd a}s\big)(g)&=\tilde\xi(\{a,s\})+\tilde\xi(\tilder\Xi)\big(g,\tilder{(\extd a)}(g),s(g)\big)\\
&\quad {}+\tilder\Xi\big(g,\tilde\xi(\tilder{\extd a})(g),s(g)\big)+\tilder\Xi\big(g,\tilder{(\extd a)}(g),\tilde\xi(s)(g)\big)\\
&=\{\xi\la a,s\}+\{a,\xi\la s\}+\big(\delta^1_\xi\la a\big)\big(\delta^2_\xi\la s\big)\\
&\quad{}+\tilde\xi(\tilder\Xi)\big(g,\tilder{(\extd a)}(g),s(g)\big)+\tilder\Xi\big(g,\tilde\xi(\tilder{\extd a})(g),s(g)\big)+\tilder\Xi\big(g,\tilder{(\extd a)}(g),\tilde\xi(s)(g)\big)\\
&=\tilder\nabla_{\extd \xi\la a}s+\tilder\nabla_{\extd a}(\xi\la s)+\big(\delta^1_\xi\la a\big)\big(\delta^2_\xi\la s\big) +\tilde\xi(\tilder\Xi)\big(g,\tilder{(\extd a)}(g),s(g)\big)\end{align*}
using that $\xi\la$ is the action of $\tilde\xi$ on functions and commutes with $\extd$, being the Lie derivative along~$\tilde\xi$ on 1-forms. Note that $\tilde\xi(\tilder\Xi)$ is $\tilde\xi$ acting on $\tilder\Xi$ as a function on $G$ (its first argument). We see that
\begin{equation}\label{covcon}
\xi\la\hat\nabla_{\extd a}\extd b=\hat\nabla_{\extd(\xi\la a)}\extd b+\hat\nabla_{\extd a}\extd(\xi\la b)+\big(\delta_\xi^1\la a\big)\extd\big(\delta_\xi^2\la b\big)
\end{equation}
holds for all $\xi\in\cg$, $a,b\in C^\infty(G)$ if and only if $\tilder\Xi$ is constant in its first argument. This extends to general 1-forms as (\ref{eqtaocov}) in a straightforward manner, the details of which are omitted. \end{proof}

\begin{Remark}\label{actionconn} Let $G$ be connected and simply connected with Lie algebra $\cg$. Let $U(\cg)[[\lambda]]$ be a~deformation of $U(\cg)$ with coproduct of the form (\ref{Deltadef}), dual to $C^\infty(G)[[\lambda]]$ as a deformation of~$C^\infty(G)$ with product $*$ as in (\ref{proddef}). That the latter is a left module algebra appears to leading order as (\ref{eqPLGprop}). That a deformation $\Omega^1(G)[[\lambda]]$ is a left and right Hopf module for the $\bullet$ products as in (\ref{proddef}) appears as (\ref{covcon}). In fact, we only need these deformations to $\mathcal{O}\big(\lambda^2\big)$ so that the required notion of covariance applies even when $\hat\nabla^G$ has curvature, as explained in Section~\ref{secpredef}.
 \end{Remark}

\begin{Remark}\label{leftinvversion} In the original left translation-invariant conventions of \cite{BegMa:sem,MaTao:plg}, our covariance condition for a PLG in Lemma~\ref{PLGprop} comes out equivalently as
\begin{equation}\label{actionPBright}
\{ a,b\}\ra\xi=\{a\ra\xi,b\}+\{a,b\ra\xi\}+\big(a\ra\delta_\xi^1\big)\big(b\ra\delta^2_\xi\big).
\end{equation}
Similarly, for a left-translation covariant differential structure, we need
\begin{equation*}%\label{actionconnright}
\hat\nabla_{\extd a}\extd b\ra\xi=\hat\nabla_{\extd(a\ra\xi)}\extd b+\hat\nabla_{\extd a}\extd(b\ra\xi)+\big(a\ra\delta_\xi^1\big)\extd\big(b\ra\delta_\xi^2\big)
\end{equation*}
or on general 1-forms
\begin{equation}\label{generalactionconnright}
(\hat\nabla_{\eta}\tau)\ra\xi=\hat\nabla_{\eta\ra\xi}\tau+\hat\nabla_{\eta}(\tau\ra\xi)+i_{\tilde{\delta_\xi^1}}(\eta) \big(\tau\ra\delta_\xi^2\big),
\end{equation}
and this holds if and only if the original $\tilde\Xi$ in (\ref{mapXi}) is constant on $G$.
\end{Remark}

If one has both left and right covariance then we say that our contravariant connection is bicovariant.

\begin{Theorem}\label{bicovXi} A right-covariant contravariant connection given by $\tilde\nabla_{\extd a}s=\{a,s\}+\Xi\big(\tilde{\extd a},s\big)$ as in \eqref{mapXi} by $\Xi\colon \cg^*\tens \cg^*\to \cg^*$ is bicovariant if and only if
\begin{equation*}%\label{eq1bicovXi}
\delta_{\mathfrak{g}^*}\Xi(\phi,\psi)-\Xi(\phi_{(1)},\psi)\tens\phi_{(2)}-\Xi(\phi,\psi_{(1)})\tens\psi_{(2)}=\psi_{(1)}\tens [\phi,\psi_{(2)}]_{\cg^*}, \end{equation*}
where $\delta_{\mathfrak{g}^*}(\phi)=\phi_{(1)}\tens\phi_{(2)}$, or equivalently in dual form,
\begin{equation*}%\label{eq2bicovXi}
 \Xi^*_{[\eta,\xi]}= \big[\Xi^{*1}_\eta,\xi\big]\tens\Xi_\eta^{*2}+ \Xi^{*1}_\eta \tens\big[ \Xi_\eta^{*2},\xi\big] +\delta^1_\xi\tens \big[\eta,\delta^2_\xi\big]
\end{equation*}
for all $\eta,\xi\in \cg$, where $\Xi^*$, $\delta\colon \cg\to \cg\tens\cg$ are written explicitly $($with a sum of such terms understood$)$.
\end{Theorem}
\begin{proof} Under the first identification $\Omega^1(G)$ with $C^\infty(G,\cg^*)$ in Section~\ref{PLGprelim}, the left action on 1-forms appears as $h\la s= L_g^*R_h^*L^*_{h^{-1}g^{-1}}(s(gh))$ which differentiates ${\extd \over\extd t}\big|_0$ for $h=e^{t\xi}$ to $(\xi\la s)(g)=\tilde\xi(s)(g)+ {\rm ad}_\xi^*(s(g))$ where ${\rm ad}_\xi(\phi)=\phi\o\<\phi\t,\xi\>$ in terms of the Lie cobracket on~$\cg^*$. Also note that $\xi\la\extd a=\extd\tilde\xi(a)$ as the action commutes with~$\extd$. Then the left-covariance in the form used in Proposition~\ref{taocovarianceprop} appears in these terms as the condition
\begin{align*} \tilde\xi(\{a,s\})&+{\rm ad}_\xi^*(\{a,s\})+ \tilde\xi\big(\Xi(\tilde{\extd a},s)\big) + {\rm ad}_\xi^*\Xi\big(\tilde{\extd a},s\big)\\
&=\{\tilde\xi(a),s\}+\Xi\big(\widetilde{\extd} \big(\tilde\xi(a)\big),s\big)+ \big\{a,\tilde\xi(s)+ {\rm ad}_\xi^*(s)\big\}+\Xi\big(\tilde{\extd a},\tilde\xi(s)+{\rm ad}_\xi^*(s)\big)\\
&\quad+\tilde{\delta^1_\xi}(a)\big(\tilde{\delta^2_\xi}(s)+{\rm ad}^*_{\delta^2_\xi}(s)\big).
\end{align*}
We now expand $\tilde\xi\{a,s\}$ using the PLG covariance in Lemma~\ref{PLGprop} (noting that $s$ is a $\cg^*$-valued function under our trivialisation of the bundle) and we expand $\tilde\xi\big(\Xi(\tilde{\extd a},s)\big)=\Xi\big(\tilde\xi\big(\tilde{\extd a}\big),s\big)+\Xi\big(\tilde{\extd a},\tilde\xi(s)\big)$ since~$\Xi$ itself is constant on~$G$, and cancel terms. Then our condition becomes
\begin{gather*}
{\rm ad}_\xi^*(\{a,s\})+ {\rm ad}_\xi^*\Xi\big(\tilde{\extd a},s\big)+ \Xi\big(\tilde\xi\big(\tilde{\extd a}\big),s\big)\\
\qquad{} =\Xi\big(\widetilde{\extd} \big(\tilde\xi(a)\big),s\big)+ \{a, {\rm ad}_\xi^*(s)\}+\Xi\big(\tilde{\extd a},{\rm ad}_\xi^*(s)\big)+\tilde{\delta^1_\xi}(a){\rm ad}^*_{\delta^2_\xi}(s).
\end{gather*}
Next, we note that ${\rm ad}_\xi^*(\{a,s\})= \{a, {\rm ad}_\xi^*(s)\}$ as ${\rm ad}^*$ acts only on the $\cg^*$ values and that $\widetilde{\extd} \big(\tilde\xi(a)\big)- \tilde\xi\big(\tilde{\extd a}\big)={\rm ad}_\xi^*\big(\tilde\extd a\big)$ deduced from formulae in the preliminaries (or by expanding $\extd a=\del_i(a)f^i$ where $\del_i=\tilde{e_i}$ in a basis of the Lie algebra with dual basis $\{f^i\}$). Using these identities, our condition for bicovariance reduces to
\[ {\rm ad}_\xi^*\Xi\big(\tilde{\extd a},s\big)=\Xi\big({\rm ad}_\xi^*\big(\tilde{\extd a}\big),s\big)+\Xi\big(\tilde{\extd a},{\rm ad}_\xi^*(s)\big)+\tilde{\delta^1_\xi}(a){\rm ad}^*_{\delta^2_\xi}(s),\]
which being true for all $s$ and all $a$ reduces to
\[ {\rm ad}_\xi^*\Xi(\phi,\psi)=\Xi({\rm ad}_\xi^*\phi,\psi)+\Xi(\phi,{\rm ad}_\xi^*\psi)+\big\langle\delta^1_\xi,\phi\big\rangle {\rm ad}^*_{\delta^2_\xi}\psi\]
for all $\phi,\psi\in\cg^*$, and for all $\xi\in\cg$. This is the first stated form of the condition in terms of the Lie cobracket (when evaluated against $\xi$ in the second tensor factor and converting the bracket on $\cg^*$ to a cobracket on $\xi$). \end{proof}

The stated characterisation of bicovariant $\Xi$ was obtained in the pre-Lie algebra flat connection case in~\cite{MaTao:plg}, but now we see that the same holds in general using our new methods.

\begin{Example}\label{exSU2} At the algebraic level, we consider the standard Hopf algebra $\mathbb{C}_q[{\rm SL}_2]$ with generators $a$, $b$, $c$, $d$ and relations
\begin{gather*} ba=qab, \qquad ca=qac, \qquad db=qbd, \qquad dc=qcd, \qquad bc=cb, \\
 ad-da=\big(q^{-1}-q\big)bc, \qquad ad-q^{-1}bc=1, \end{gather*} and a standard matrix form of coproduct and antipode~\cite{Ma:book}. There is a left translation covariant $\Omega^1(\mathbb{C}_q[{\rm SU}_2])$ with~\cite{Wor}
\[ e^0=d\extd a-qb\extd c,\qquad e^+=q^{-1}a\extd c-q^{-2}c\extd a,\qquad e^-=d\extd b-qb\extd d \]
a left basis of left-invariant 1-forms and the right module structure given by the bimodule relations
\[ e^0f=q^{2|f|}fe^0, \qquad e^{\pm}f=q^{|f|}e^{\pm},\]
for homogeneous $f$ of degree $|f|$, where $|a|=|c|=1$ and $|b|=|d|=-1$. The exterior derivative is
\[ \extd a=ae^0+qbe^+,\qquad \extd b=ae^--q^{-2}be^0,\qquad \extd c=ce^0+qde^+,\qquad \extd d=ce^--q^{-2}de^0.\]
We now set $q=e^{\frac{\lambda}{2}}$, for $\lambda$ a deformation parameter and work at the corresponding semiclassical level albeit focussing on the polynomial subalgebra generators. Classically, the basis $e^0$, $e^\pm$ of 1-forms is dual to the basis $\big\{\tilde H,\tilde X_\pm\big\}$ of left-invariant vector fields generated by the Chevalley basis $\lbrace H,X_{\pm}\rbrace$ of $\mathfrak{su}_2$. The latter can be read off from $\extd f=\tilde{H}(f)e^0+\tilde{X_\pm}(f)e^\pm$ (sum over $\pm$) as
\[ \tilde H\begin{pmatrix} a & b \\ c & d \end{pmatrix}=\begin{pmatrix} a & -b \\ c & -d \end{pmatrix},\qquad \tilde X_+\begin{pmatrix} a & b \\ c & d \end{pmatrix}=\begin{pmatrix} 0 & a \\ 0 & c \end{pmatrix},\qquad \tilde X_-\begin{pmatrix} a & b \\ c & d \end{pmatrix}=\begin{pmatrix} b & 0 \\ d & 0 \end{pmatrix}. \]
The left coaction covariance of the calculus corresponds in Remark~\ref{leftinvversion} to covariance under a~right action of the Lie algebra, with right-invariant vector fields
\[ \tilder H\begin{pmatrix} a & b \\ c & d \end{pmatrix}=\begin{pmatrix} a & b \\ -c & -d \end{pmatrix},\qquad \tilder X_+\begin{pmatrix} a & b \\ c & d \end{pmatrix}=\begin{pmatrix} c & d \\ 0 & 0 \end{pmatrix},\qquad \tilder X_-\begin{pmatrix} a & b \\ c & d \end{pmatrix}=\begin{pmatrix} 0 & 0 \\ a & b \end{pmatrix}. \]

(1) From the algebra relations, the Poisson bracket can be read off on the generators as
\begin{gather*} \lbrace a,b\rbrace=-\tfrac{1}{2} ab,\qquad\! \lbrace a,c\rbrace=-\tfrac{1}{2}ac,\qquad\! \lbrace a,d\rbrace=-bc, \qquad\! \lbrace b,d\rbrace=-\tfrac{1}{2}bd,\qquad\! \lbrace c,d\rbrace=-\tfrac{1}{2}cd\end{gather*}
and $\lbrace b,c\rbrace=0$, from which one can easily verify~(\ref{actionPBright}) as a check on calculations, with $\delta(X_\pm)={1\over2}(X_\pm\tens H-H\tens X_\pm)$ and $\delta(H)=0$.

(2) From the bimodule relations, \cite{MaTao:plg} provides the Poisson-compatible contravariant connection as
\[ \hat\nabla_{\extd \left(\begin{smallmatrix} a & b \\ c & d \end{smallmatrix}\right)}e^i={1\over 2}t_i \begin{pmatrix} a & -b \\ c & -d \end{pmatrix}e^i,\qquad t_0=-2,\qquad t_\pm=-1, \]
where $i\in\lbrace 0,\pm\rbrace$, from which one can easily verify~(\ref{generalactionconnright}) as a check.
\end{Example}

\section{PLG actions on Poisson manifolds}\label{PLGact}

As a step towards principal bundles, we now `polarise' the above to a general Poisson mani\-fold~$X$ right acted upon by a Lie group $G$ with action
\[ \beta\colon \ X\times G\to X,\qquad \beta(x,g)=x.g,\qquad\forall\, x\in X, g\in G.\]
As before, we work in a coordinate algebra language with $P=C^\infty(X)$ with action $(g\la f)(x)=f(x.g)$ for all $x\in X$, $g\in G$ and $f\in P$. If we have a Hopf algebraic version $H=\C[G]$ of $C^\infty(G)$ then correspondingly $P$ would be a right $H$-comodule algebra with coaction $\Delta_Rp=p\rz\tens p\co\in P\tens H$.

As in Drinfeld's theory, we now suppose further that $G$ is a PLG and ask that $\beta$ is a Poisson map where $X\times G$ has the direct product Poisson structure~\cite{Lu}, which we refer to as a {\em PLG action}. Note that in classical Poisson geometry, where~$G$ is just a Lie group, it is more normal to consider each map $\beta(\ ,g)\colon X\to X$ as a Poisson map, which effectively would mean the zero Poisson bracket on~$G$.

Adopting a similar notation as in the preceding section, we let $\pi_X$ and $\pi_G$ be the respective Poisson tensors and we define
\[ \beta_g\colon \ X\to X,\qquad \beta_x\colon \ G\to X,\qquad \beta_g(x)=\beta_x(g)=x.g\]
with differentials
\[ \beta_{g*}\colon \ T_xX\rightarrow T_{x.g}X, \qquad \beta_{x*}\colon \ T_gG\rightarrow T_{x.g}X, \]
and extensions to act on tensor products in the usual way. Then the condition for a PLG action is clearly
\begin{equation*}%\label{PLGactiondef}
\pi_X(x\cdot g)=\beta_{g*}\pi_X(x)+\beta_{x*}\pi_G(g), \end{equation*}
for all $x\in X$, $g\in G$. Following our previous notation, we also denote by $\tilde\xi=\xi\la$ the vector field associated to $\xi\in\cg$ for the infinitesimal action on $X$. This is the left action of $g=e^{t\xi}$ on functions differentiated in~$t$ at $t=0$, i.e., $\tilde\xi(x)=\beta_{x*}(\xi)$ if we view $\cg=T_eG$. Since these vector fields are defined for each element~$\xi$, we can think of them all together as a single map
\begin{equation*}%\label{ver}
\ver\colon \ \Omega^1(X)\rightarrow C^{\infty}(X)\tens\mathfrak{g}^*\end{equation*} such that when evaluated against element $\xi$ of the Lie algebra we recover $\tilde\xi$ in the form $\ver_\xi=(\id\tens\xi)\ver=i_{\tilde\xi}\colon \Omega^1(X)\to C^\infty(X)$. Our first observation is a known result but cast in a language closer to a quantum group module algebra.

\begin{Lemma}[{cf.~\cite[p.~54]{KS}}]\label{PLGactionprop}
Let $G$ be a connected Poisson--Lie group, $X$ a Poisson manifold and $\beta$ a smooth right action of the group~$G$ on~$X$. This is a PLG action if and only if
 \begin{equation}\label{eqPLGactionprop}
\xi\triangleright\lbrace p,q\rbrace=\lbrace \xi\triangleright p,q\rbrace+\lbrace p,\xi\triangleright q\rbrace+\big(\delta_\xi^1\triangleright p\big)\big(\delta_\xi^2\triangleright q\big)
\end{equation}
holds for all $p,q\in C^\infty(X)$, $\xi\in \cg$. \end{Lemma}
\begin{proof} This follows the same pattern as in the proof of Lemma~\ref{PLGprop} and is, moreover, equivalent to~\cite[p.~54]{KS}. Hence we omit any details.
\end{proof}

There is clearly an analogous analysis to \cite{BegMa:sem} as to the natural notion of a Poisson-compatible contravariant connection $\hat\nabla^X$ on $X$ being compatible, polarising the case of $G$ acting on itself with contravariant connection~$\hat\nabla^G$. Without giving such an analysis here, we take the polarised version of~(\ref{eqtaocov}) as a definition, i.e.,
\begin{Definition}\label{concovX} Let $G$ be a connected PLG on a PLG-covariant Possion manifold~$X$. A contravariant connection $\hat\nabla$ on $\Omega^1(X)$ is similarly {\em $G$-covariant} if and only if
\begin{equation}\label{eq2concovX}
\xi\la\hat{\nabla}_{\eta}\tau=\hat{\nabla}_{\xi\la\eta}\tau+\hat{\nabla}_{\eta}(\xi\la\tau)+(i_{\tilde{\delta_\xi^1}}\eta) \big(\delta_\xi^2\la\tau\big),
\end{equation}
for all $\xi\in \cg$ and $\eta,\tau\in \Omega^1(X)$.
\end{Definition}

\begin{Remark}\label{remomegaXcov} One can check that this is compatible with Lemma~\ref{PLGactionprop} and the axioms of a contravariant connection in the sense that $\xi\la\hat\nabla_{\extd p}(q\tau)=\xi\la\big(\{p,q\}\tau+q\hat\nabla_{\extd p}\tau\big)$ expanded using~(\ref{eq2concovX}) on the left hand side agrees with~(\ref{eqPLGactionprop}) and $\xi\la\hat\nabla_{\extd p}\tau$ expanded separately. It similarly respects Poisson-compatibility, e.g., on exact forms we have
\begin{equation*}%\label{eq1concovX}
\xi\triangleright\hat{\nabla}_{\extd p}\extd q=\hat{\nabla}_{\extd(\xi\triangleright p)} \extd q+\hat{\nabla}_{\extd p}\extd(\xi\triangleright q)+\big(\delta_\xi^1\triangleright p\big)\extd\big(\delta_\xi^2\triangleright q\big), \end{equation*}
for all $\xi\in\cg$ and $p,q\in C^\infty(X)$. Subtracting the same with $p\leftrightarrow q$ gives $\extd\{p,q\}$ on using~(\ref{eqPLGactionprop}), antisymmetry of the output of $\delta_\xi$ and the Poisson-compatibility applied to $\extd\{\xi\la p,q\}$ and $\extd\{p,\xi\la q\}$ separately. The parallel of Remark~\ref{actionconn} also applies. If $U(\cg)[[\lambda]]$ acts on $C^\infty(X)[[\lambda]]$ in a deformation context then $\xi\la(p * q-q*p)$ computed as a module algebra gives the PLG action condition~(\ref{eqPLGactionprop}) from the order $\lambda$ terms. Similarly, if this action extends to a deformed calculus $\Omega^1(X)[[\lambda]]$ then $\xi\la(p\bullet\extd q-\extd q\bullet p)$ computed as a module algebra requires~(\ref{eq2concovX}). This then generalises to arbitrary 1-forms as~(\ref{eq2concovX}). In fact, we only need these deformations to $\mathcal{O}\big(\lambda^2\big)$ as our contravariant connections could be curved.
\end{Remark}

To explain the next corollary, note that the Lie derivative can be formally extended to one along antisymmetric tensors $V$ by the formula $\CL_V=i_V\extd -(-1)^{|V|} \extd i_V$ where $i_V\colon \Omega(X)\to \Omega(X)$ lowers degree by that of~$V$. In particular, $i_{v\wedge w}=i_v i_w$ (by which we mean a sum of such terms) depends antisymmetrically on $v$, $w$ and so descends to the wedge product, giving a well-defined degree $-2$ interior product on the exterior algebra. This in turn defines a Lie derivative $\CL_{v\wedge w}=[i_{v\wedge w},\extd]$ along antisymmetric bivector fields. We will also need, as in~\cite{Ma:rec}, the Leibnizator $L_B(\tau,\eta)=B(\tau\wedge\eta)-(B\tau)\wedge\eta-(-1)^{|B||\tau|}\tau\wedge B\eta$ for $B$ an operator on the exterior algebra of degree $|B|$.

\begin{Corollary}\label{corconcovX} Let $\hat\nabla$ be a PLG-covariant contravariant connection on $\Omega^1(X)$ in the sense of Definition~{\rm \ref{concovX}}. If this is Poisson-compatible then
\begin{equation*}%\label{eqcorconcovX}
\xi\la[\tau,\eta]=[\xi\la\tau,\eta]+[\tau,\xi\la\eta]+\tfrac{1}{2} L_{\CL_{\tilde{\delta\xi}}}(\tau,\eta)\end{equation*}
also holds for all $\eta,\tau\in \Omega^1(X)$ and $\xi\in \cg$.
\end{Corollary}
\begin{proof}
Recall that Poisson-compatibility of the connection is $[\tau,\eta]=\hat\nabla_\tau\eta-\hat\nabla_\eta\tau$. Hence in the covariant case,
\begin{align*}
\xi\la[\tau,\eta]& =\xi\la(\hat\nabla_\tau\eta-\hat\nabla_\eta\tau)\\
&=\hat\nabla_{\xi\la\tau}\eta+\hat\nabla_\tau(\xi\la\eta)+\big(i_{\tilde\delta_\xi^1}\tau\big) \big(\delta^2_\xi\la\eta\big)-\hat\nabla_{\xi\la\eta}\tau-\hat\nabla_\eta(\xi\la\tau)-\big(i_{\tilde\delta_\xi^1}\eta\big) \big(\delta^2_\xi\la\tau\big)\\
& =[\xi\la\tau,\eta]+[\tau,\xi\la\eta]+\big(i_{\tilde\delta_\xi^1}\tau\big) \big(\delta^2_\xi\la\eta\big)-\big(i_{\tilde\delta_\xi^1}\eta\big)\big(\delta^2_\xi\la\tau\big).
\end{align*}
Moreover, if $v\wedge w$ is an antisymmetric product of vector fields on $X$ and $\tau$, $\eta$ are 1-forms, we have
\begin{align*}
\CL_{v\wedge w}(\tau\wedge\eta)&=i_vi_w(\extd\tau\wedge \eta-\tau\wedge\extd\eta)-\extd(i_{v}i_w(\tau\wedge\eta))\\
&=i_{v\wedge w}(\extd \tau)\eta-i_v(\eta)i_w(\extd \tau)+i_v(\extd \tau)i_w(\eta)-i_v(\extd\eta)i_w(\tau)\\
&\quad{}+i_v(\tau)i_w(\extd\eta)-i_{v\wedge w}(\extd\eta)\tau -\extd(i_v(\eta)i_w(\tau)-i_v(\tau)i_w(\eta))\\
&=i_{v\wedge w}(\extd \tau)\eta -i_{v\wedge w}(\extd\eta)\tau + 2 i_v(\tau)(i_w(\extd \eta)+\extd i_w(\eta))+2 i_w(\eta)(i_v(\extd \tau)+\extd i_v(\tau))\!\\
&=i_{v\wedge w}(\extd \tau)\eta -i_{v\wedge w}(\extd\eta)\tau + 2 i_v(\tau)\CL_w\eta+2 i_w(\eta)\CL_v\tau \end{align*}
so that
\[ i_v(\tau)\CL_w\eta+ i_w(\eta)\CL_v\tau=\tfrac{1}{2} \big( \CL_{v\wedge w}(\tau\wedge\eta)-( i_{v\wedge w}\extd \tau)\eta+ \tau (i_{v\wedge w}\extd\eta)\big)=\tfrac{1}{2} L_{\CL_{v\wedge w}}(\tau,\eta)\]
allowing us to write the $\delta_\xi$ terms as stated in terms of the bivector $\tilde{\delta_\xi}=\tilde{\delta_\xi^1}\wedge\tilde{\delta_\xi^2}$.
\end{proof}

This formula reduces to $\extd$ applied to the $\xi\la\{a,b\}$ covariance identity if $\tau$, $\eta$ are exact. It also extends in principle to higher degree forms.

\begin{Remark}\label{RemarkPLGactright} We also have a right-covariance condition for a PLG left action $\alpha\colon G\times X\to X$, namely
\begin{equation*}%\label{PLGactionPBright}
\lbrace p,q\rbrace\ra\xi=\lbrace p\ra\xi,q\rbrace+\lbrace p,q\ra\xi\rbrace+\big(p\ra\delta_\xi^1\big)\big(q\ra\delta_\xi^2\big),
\end{equation*}
where $\ra\xi$ on functions is the vector field $\tilder\xi$ for the infinitesimal action. The covariance of the calculus and the consequence of Poisson-compatibility now appear as
\begin{gather*} (\hat{\nabla}_{\eta}\tau)\ra\xi=\hat{\nabla}_{\eta\ra\xi}\tau +\hat{\nabla}_{\eta}(\tau\ra\xi)+(i_{\tilder{\delta_\xi^1}}\eta)\big(\tau\ra\delta_\xi^2\big),\\
 [\tau,\eta]\ra\xi=[\tau\ra\xi,\eta]+[\tau,\eta\ra\xi]+\tfrac{1}{2} L_{\CL_{\tilder{\delta \xi}}}(\tau,\eta).
 \end{gather*}
\end{Remark}

\begin{Example}\label{exS2PB}
Here we consider the right coaction $\Delta_R(f)=f\tens t^{|f|}$ of $H=\C_{q^2}\big[S^1\big]=\C\big[t,t^{-1}\big]$ with $\extd t.t=q^2t\extd t$ on $f\in P=\C_q[{\rm SU}_2]$ with its calculus and notations as in Example~\ref{exSU2}. At the semiclassical level, $G=S^1$ acting from the right on $X={\rm SU}_2$ (as a diagonal subgroup) as a PLG with the zero Poisson bracket, so $\delta=0$. From the calculus, $\extd t.t=(1+{\lambda})t.\extd t+\mathcal{O}\big(\lambda^2\big)$ gives the Poisson-compatible contravariant connection
\begin{equation*}%\label{conS1}
\hat\nabla^G_{\extd t}\extd t=-t\extd t\end{equation*}
on $\Omega^1\big(S^1\big)$. This has left-invariant basic 1-form $t^{-1}\extd t$ with dual $\tilde\xi=t{\del\over{\del t}}$ as a left-invariant vector field. Then $\hat\nabla_{t^{-1}\extd t}\big(t^{-1}\extd t\big)=-t^{-1}\extd t$, which translates to
\begin{equation}\label{conS2} \Xi_\xi^*=-\xi\tens\xi.\end{equation}
It is also clear from the form of the coaction that $\tilde{\xi}(f)=|f|f$ so that $\tilde\xi=H\la(\ )=\tilde H$ as a~vector field for the infinitesimal action on ${\rm SU}_2$ using the partial derivative in the $e^0$ direction as given in Example~\ref{exSU2}. The Poisson bracket on ${\rm SU}_2$ in Example~\ref{exSU2} is then easily seen to be right covariant for the action of~$S^1$. This is well-known and details are omitted.
For the covariance condition in Definition~\ref{concovX}, we compute $H\la e^0=0$, $H\la e^+=2e^+$, $H\la e^-=-2e^-$. Then, for example
\begin{align*}H\la\hat\nabla_{\extd\left(\begin{smallmatrix} a & b \\ c & d \end{smallmatrix}\right)}e^+&=H\la\left({1\over 2}t_+ \begin{pmatrix} a & -b \\ c & -d \end{pmatrix}e^+\right)=-{1\over 2} \begin{pmatrix} a & b \\ c & d \end{pmatrix}e^+- \begin{pmatrix} a & -b \\ c & -d \end{pmatrix}e^+\\
&=\hat\nabla_{\extd\left(H\la\left(\begin{smallmatrix} a & b \\ c & d \end{smallmatrix}\right)\right)}e^++\hat\nabla_{\extd\left(\begin{smallmatrix} a & b \\ c & d \end{smallmatrix}\right)}(H\la e^+)\end{align*}
as required since $\delta=0$. \end{Example}

\section{Formulation of Poisson principal bundles}\label{PPB}

A principal bundle in a smooth setting is a smooth manifold $X$ with a smooth, free and proper action of a Lie group $G$ and a local triviality condition so that $M=X/G$ is a smooth manifold and $X$ a fibre bundle over it with fibre $G$. We have discussed actions and now we consider the further data we need for a principal bundle at the quantum and hence semiclassical level. From a~practical perspective, the key expression of local triviality is transversality: the $C^\infty(X)$-module of horizontal forms $\Omega^1_{\rm hor}$ (the pull back of $\Omega^1(M)$ along the canonical projection $X\to M$) is precisely the joint kernel of the vector fields for the infinitesimal action. The invariant horizontal forms are then the forms on the base, $\Omega^1(M)\hookrightarrow\Omega^1_{\rm hor}(X)$. In the quantum case, the transversality is exactness of the sequence~(\ref{PHexact}), the horizontal forms are $P\Omega^1_AP$ with (under reasonable conditions) invariants under the coaction of $H$ recovering the quantum calculus $\Omega^1_A\subseteq \Omega^1_P$. Therefore the main missing ingredient we need is the semiclassical version of this transversality condition. Motivated by the assumption of $*$-product quantisations, see Remark~\ref{defbunrem}, we are led to the following.
\begin{Definition}\label{PPBdef}
A {\em Poisson principal bundle} means a classical principal bundle $X\to X/G$ with
\begin{enumerate}\itemsep=0pt
\item $X$ a Poisson manifold and $G$ a Poisson--Lie group such that the action on $X$ is a Poisson--Lie group action of $G$ in the sense of Lemma~\ref{PLGactionprop}.
\item A bicovariant Poisson-compatible connection on $G$ given by $\Xi\colon \cg^*\tens\cg^*\to\cg^*$.
\item A Poisson-compatible connection on $X$ which is left $G$-covariant in the sense of Definition~\ref{concovX}.
\item A {\em Poisson transversality condition} for all $\eta,\tau\in \Omega^1(X)$ and $\xi\in\cg$,
\begin{equation*}
i_{\tilde\xi}\big(\hat\nabla^X_{\eta}\tau\big)=i_{\widetilde{\Xi^{*1}_\xi}}(\eta)i_{\widetilde{\Xi^{*2}_\xi}}(\tau)+\pi_X(\eta,\extd i_{\tilde\xi}(\tau)).
\end{equation*}
\end{enumerate}
\end{Definition}

In principle, we could need more conditions, but we will see that the above is sufficient for our purposes.

\begin{Corollary}\label{corintsch} Let $\hat\nabla$ be a left-covariant contravariant connection $\Omega^1(X)$ obeying the transversality condition $(4)$ in Definition~{\rm \ref{PPBdef}}. If this is Poisson-compatible then
\begin{equation*}%\label{eqcortrans}
i_{\tilde\xi}[\eta,\tau]= \pi_X(\eta,\extd i_{\tilde\xi}(\tau))-\pi_X(\extd i_{\tilde\xi}(\eta),\tau)-\tfrac{1}{2} i_{\tilde{\delta\xi}}(\eta\wedge\tau)\end{equation*}
for all $\eta,\tau\in \Omega^1(X)$.
\end{Corollary}
\begin{proof} We compute using Poisson-compatibility,
\begin{align*}
i_{\tilde\xi}[\eta,\tau]& =i_{\tilde\xi}\hat{\nabla}^X_{\eta}\tau-i_{\tilde\xi}\hat{\nabla}^X_{\tau}\eta\\
& =i_{\widetilde{\Xi^{*1}_\xi}}(\eta)i_{\widetilde{\Xi^{*2}_\xi}}(\tau)+\pi_X(\eta,\extd i_{\tilde\xi}(\tau))- i_{\widetilde{\Xi^{*1}_\xi}}(\tau)i_{\widetilde{\Xi^{*2}_\xi}}(\eta)-\pi_X(\tau,\extd i_{\tilde\xi}(\eta))\\
&= \pi_X(\eta,\extd i_{\tilde\xi}(\tau))-\pi_X(\extd i_{\tilde\xi}(\eta),\tau)+i_{(\tilde{\delta\xi)^1}}(\eta)i_{(\tilde{\delta\xi)^2}}(\tau).\end{align*}
Here, \eqref{Xipoisson} in dual form allows us to recognise $\tilde{\delta\xi}$ from the cocommutator. We also use antisymmetry of $\pi_X$. We then put the result as a bivector interior product as discussed before.
\end{proof}

On exact differentials, this reduces to the covariance condition on $\xi\la\{p,q\}$ for the Poisson bracket. We also want to know that our definition is fit for purpose and implies that the base is not only a~manifold but a Poisson manifold with an induced Poisson-compatible contravariant connection.

\begin{Proposition}\label{basecalc} Let $G$ be a PLG with a free smooth proper right PLG action on a Poisson manifold $X$, equipped with Poisson-compatible contravariant connections $\hat\nabla^G$ and $\hat\nabla^X$ so as to form a Poisson-principal bundle as in Definition~{\rm \ref{PPBdef}}. Then $M=X/G$ becomes a Poisson manifold with Poisson-compatible contravariant connection $\hat\nabla^M$ the restriction of $\hat\nabla^X$.
\end{Proposition}
\begin{proof} Here $C^\infty(M)$ can be identified with smooth functions on $X$ which are killed by all vector fields for the infinitesimal action. By the invariance of the Poisson bracket, it follows that if $p$, $q$ are killed by all $\tilde\xi$ then so is $\{p,q\}$, so the Poisson bracket restricts. Now let $\hat\nabla^M$ be the restriction of $\hat\nabla^X$ to the differentials of such functions. By the covariance condition, it follows that the output of $\hat\nabla^M$ is invariant under all $\xi\la=\CL_{\tilde\xi}$. By the bundle condition, it follows that the output of $\hat\nabla^M$ is killed by $i_{\tilde\xi}$. However, if $\sum p_i\extd q_i$ with $p_i,q_i\in C^\infty(X)$ has these properties then $\sum p_i\tilde{\xi}(q_i)=0$ and moreover, by transversality of the classical bundle, we know that the form is horizontal, so we can assume $q_i\in C^\infty(M)$. Then by the first property, $\sum\tilde\xi(p_i)\extd q_i=0$ so that taking the $\extd q^i$ independent, the $p_i$ are also in $C^\infty(M)$, i.e., our form is an element of~$\Omega^1(M)$. Thus $\hat\nabla^M$ is defined, and in this case inherits the connection properties. \end{proof}

\begin{Example}\label{exspherebun} Following on from Example~\ref{exS2PB} for the Poisson version of ${\rm SU}_2\to S^2={\rm SU}_2/S^1$, one can verify that the Poisson brackets and connections there obey the transversality condition~(4) in Definition~\ref{PPBdef}. For example,
\begin{align*}
i_{\tilde\xi}\Big(\hat\nabla^X_{\extd \left(\begin{smallmatrix} a & b \\ c & d\end{smallmatrix}\right)}e^0\Big)&=i_{\tilde\xi}\left( \begin{pmatrix} -a & b \\ -c & d\end{pmatrix}e^0\right)= \begin{pmatrix} -a & b \\ -c & d\end{pmatrix}\\
&=-\tilde{\xi}\left( \begin{pmatrix} a & b \\ c & d\end{pmatrix}\right)i_{\tilde\xi}\big(e^0\big)+\left\{ \begin{pmatrix} a & b \\ c & d\end{pmatrix},i_{\tilde\xi}e^0\right\}_X\\
&=i_{\widetilde{\Xi_\xi^{*1}}}\left(\extd \begin{pmatrix} a & b \\ c & d\end{pmatrix}\right)i_{\widetilde{\Xi_\xi^{*2}}}\big(e^0\big)+\pi_X\left(\extd \begin{pmatrix} a & b \\ c & d\end{pmatrix},\extd\big(i_{\tilde\xi}e^0\big)\right),\end{align*}
where $\xi=H$ and $\tilde\xi$ is the vector field for its infinitesimal action by right-translation as in Example~\ref{exSU2}. This is dual to $e^0\in\Omega^1$, so $\tilde\xi(a)=a$ and $i_{\tilde\xi}\big(e^0\big)=\big\langle \tilde\xi,e^0\big\rangle=1$. We also used $\Xi_\xi^*=-\xi\tens\xi$ from~(\ref{conS2}).

Hence we obtain a Poisson bracket and connection on $M=S^2/S^1$ by Proposition~\ref{basecalc}, which we now describe. If we use the complexified coordinates $z=cd$, $z^*=-ab$, $x=-bc$ then the algebra relations are given by
\[ zx=(1+\lambda)xz, \qquad z^*x=(1-\lambda)xz^*, \qquad zz^*=(1+2\lambda)z^*z-\lambda x, \qquad z^*z=x(1-x).\]
From which, the Poisson bracket and Poisson-compatible contravariant connection can be computed as
\begin{gather*}
 \{ z,x \}_{S^2}=xz, \qquad \{ z^*,x \}_{S^2}=-xz^*, \qquad \{ z,z^* \}_{S^2}=2z^*zx, \\
\hat\nabla^{S^2}_{\extd z}\extd z=(2x-1)z\extd z-2z^2\extd x, \qquad \hat\nabla^{S^2}_{\extd z}\extd z^*=-(2x-1)z\extd z^*-2x^2\extd x,\\
\hat\nabla^{S^2}_{\extd z^*}\extd z=(2x-1)z^*\extd z+2x^2\extd x, \qquad \hat\nabla^{S^2}_{\extd z^*}\extd z^*=-(2x-1)z^*\extd z^*+2{z^*}^2\extd x, \\
 \hat\nabla^{S^2}_{\extd z}\extd x=-(2x-1)z\extd x+(2x-1)x\extd z, \qquad \hat\nabla^{S^2}_{\extd z^*}\extd x=(2x-1)z^*\extd x-(2x-1)x\extd z^*,\\
 \hat\nabla^{S^2}_{\extd x}\extd x=(2x-1)x\extd x+2xz\extd z^*.\end{gather*}
In the present case, we also have left-translation covariance of the Poisson bracket and a~3D calculus in Example~\ref{exSU2} on ${\rm SU}_2$ commuting with the right action of~$S^1$. This necessarily descends to an action of ${\mathfrak{su}}_2$ on the sphere as
\begin{gather*} x\ra H=0, \qquad z\ra H= -2z, \qquad z^*\ra H=2z^*,\qquad x\ra X_+=-z,\qquad z\ra X_+= 0,\\
z^*\ra X_+=2x-1,\qquad x\ra X_-=z^*, \qquad z\ra X_-=1-2x, \qquad z^*\ra X_-=0\end{gather*}
with respect to which $\{\ ,\ \}_{S^2}$ and $\hat\nabla^{S^2}$ are covariant in our Poisson sense of Remark~\ref{RemarkPLGactright}. For example,
\begin{align*} \big(\hat\nabla^{S^2}_{\extd z}\extd z^*\big)\ra X_+&=\big({-}(2x-1)z\extd z^*-2x^2\extd x\big)\ra X_+=2z^2\extd z^*+2z\extd x+2x^2\extd z\\
&=4z(1-x)\extd x+2x(2x-1)\extd z\end{align*}
since $\extd x=-{1\over{2x-1}}(z^*\extd z+z\extd z^*)$ classically. On the other hand,
\begin{gather*} \hat\nabla^{S^2}_{\extd(z\ra X_+)}\extd z^* +\hat\nabla_{\extd z}\extd(z^*\ra X_+)=2\hat\nabla^{S^2}_{\extd z}\extd x=-2(2x-1)z\extd x+2x(2x-1)\extd z, \\
 \big(z\ra\delta^1_{X_+}\big) \extd\big(z^*\ra\delta^2_{X_+}\big)=\tfrac{1}{2} (z\ra X_+)\extd(z^*\ra H)-\tfrac{1}{2} (z\ra H)\extd(z^*\ra X_+)=2z\extd x, \end{gather*}
from which we see that the covariance condition holds.
\end{Example}

It remains to explain how the proposed semiclassical transversality condition~(4) relates to the quantum transversality. This should be considered purely as motivation in that we have not constructed the assumed $*$-product quantisations.

\begin{Remark}\label{defbunrem} We suppose $*$-product quantisations $C^\infty(X)[[\lambda]]$, $C^\infty(M)[[\lambda]]$ and $C^\infty(G)[[\lambda]]$ forming a quantum principal bundle when working over $\C[[\lambda]]$, at least to $\mathcal{O}\big(\lambda^2\big)$. Here $X$, $G$ are a Poisson manifold and a PLG respectively, and $\hat\nabla^G$ is bicovariant as defined by $\Xi^*$ while $\hat\nabla^X$ is covariant for the differentials of the total space of the quantum bundle. We compute $\text{ver}(p\bullet\extd q)$ in (\ref{PHver}) working at the quantum group level. We assume that the coproduct is undeformed and write $\varpi^\bullet\pi_\eps(h)=S^\bullet h\o\bullet \extd h\t$ for $S^\bullet$ the deformed antipode. Then
\begin{gather*} \lambda\text{ver}\big(\hat{\nabla}^X_{\extd p}\extd q\big) +\mathcal{O}\big(\lambda^2\big)=\text{ver}(p\bullet\extd q-(\extd q)\bullet p)=\text{ver}(p\bullet\extd q-\extd(q\bullet p)+q\bullet \extd p)\\
\qquad{} = p* q\rz \tens \varpi^\bullet\pi_\eps\big(q\co\big)-q\rz * p\rz \tens S^\bullet p\co\o \bullet \varpi^\bullet\pi_\eps\big(q\co\big)\bullet p\co\t \\
\qquad{} = p* q\rz \tens \varpi^\bullet\pi_\eps\big(q\co\big)-q\rz * p\rz \tens S^\bullet p\co\o * p\co\t\bullet \varpi^\bullet\pi_\eps\big(q\co\big) \\
\qquad\quad{} -q\rz * p\rz \tens S^\bullet p\co\o \bullet\big[ \varpi^\bullet\pi_\eps(q\co),p\co\t\big]_\bullet\\
\qquad{} = \big[p,q\rz\big]_*\tens \varpi^\bullet\pi_\eps\big(q\co\big) -q\rz * p\rz \tens S^\bullet p\co\o \bullet\big[ \varpi^\bullet\pi_\eps\big(q\co\big),p\co\t\big]_\bullet\\
\qquad{} = \lambda\big(\big\{p,q\rz\big\}\tens \varpi\pi_\eps\big(q\co\big) +q\rz p\rz \tens S p\co\o \hat\nabla^G_{{\rm d}p\co\t} \varpi\pi_\eps\big(q\co\big)\big),
\end{gather*}
where we dropped the bullets in the last line as there is already a $\lambda$ out front and we are working to $\mathcal{O}\big(\lambda^2\big)$. Hence at order $\lambda$ we must have that
\begin{align*}
\text{ver}\hat{\nabla}^X_{\extd p}\extd q&=\big\{p,q\rz\big\}\tens \varpi\pi_\eps\big(q\co\big) +q\rz p\rz \tens \hat\nabla^G_{\varpi\pi_\eps(p\co)} \varpi\pi_\eps\big(q\co\big)\\
& =\big\lbrace p, \text{ver}(\extd q)^1\big\rbrace\tens\text{ver}(\extd q)^2+\text{ver}(\extd p)^1\text{ver}(\extd q)^1\tens \hat{\nabla^G}_{\text{ver}(\extd p)^2}\text{ver}(\extd q)^2,
\end{align*}
where we brought $Sp\co\o$ into the subscript of $\hat\nabla^G$. We then recognised $q\rz\tens\varpi\pi_\eps(q\co)=\text{ver}(\extd q)=\text{ver}(\extd q)^1\tens\text{ver}(\extd q)^2$, say.

Next, recall that $\hat\nabla^G$ is defined by $\Xi\colon \mathfrak{g}^*\tens\mathfrak{g}^*\rightarrow \mathfrak{g}^*$, with dual $\Xi^*\colon \mathfrak{g}\rightarrow\mathfrak{g}\tens\mathfrak{g}$. We also note that $\ver^2$ is a left-invariant 1-form which, under the identification of $\Omega^1(G)$ with $C^\infty(G,\cg^*)$, corresponds to a constant function. Hence, our condition on $\ver \hat\nabla^X$ is equivalent to
\begin{align*}
i_{\tilde\xi}\big(\hat{\nabla}^X_{\extd p}\extd q\big)&= (\id\tens\xi)\ver \big(\hat{\nabla}^X_{\extd p}\extd q\big)\\
&=\ver (\extd p)^1\ver (\extd q)^1\tens\big\langle\Xi^*(\xi)^1,\ver (\extd p)^2\big\rangle
\big\langle\Xi^*(\xi)^2,\ver (\extd q)^2\big\rangle+ \lbrace p,\ver _\xi(\extd q) \rbrace
\end{align*}
for all $\xi\in\cg$, where $(\id\tens\xi)\ver =i_{\tilde\xi}$. Finally, $\ver (\extd p)^1\tens\big\langle\Xi^{*1},\ver (\extd p)^2\big\rangle=\ver _{\Xi^{*1}}(\extd p)=i_{\widetilde{\Xi^{*1}}}(\extd p)=\widetilde{\Xi^{*1}}(p)$, and similarly for the other factor, gives us the condition
\begin{equation*}%\label{eq1PPBlem}
i_{\tilde\xi}\big(\hat{\nabla}^X_{\extd p}\extd q\big)=\widetilde{\Xi^{*1}_\xi}(p)\widetilde{\Xi^{*2}_\xi}(q)+\big\lbrace p,{\tilde\xi}(q)\big\rbrace_X
\end{equation*}
for all $p,q\in C^\infty(X)$, $\xi\in\mathfrak{g}$. It is then straightforward to extend this to 1-forms $\tau,\eta\in\Omega^1(X)$ to find the condition (4) in Definition~\ref{PPBdef}. \end{Remark}

\section{Connections on Poisson principal bundles}\label{Assocsec}

In noncommutative differential geometry, once we have a quantum principal bundle $P$, the next order of business is to find a `spin connection' $\omega\colon \Lambda^1_H\to \Omega^1_P$ which is equivariant (where $H$ coacts on $\Lambda^1_H$ by a right adjoint coaction that follows from the assumed $\Omega^1_H$ bicovariance) and such that $\ver(\omega(v))=1\tens v$ for all $v\in\Lambda^1_H$. This provides a splitting of~$\Omega^1_P$ in the form of an associated projection $\Pi_\omega$ defined by $\Pi_\omega(\extd p)=p\rz\omega\big(\varpi\pi_\eps p\co\big)$, which is equivariant for the right coaction of~$H$, a left $P$-module map, idempotent and has $\ker\Pi_\omega=P\Omega^1_AP$, see~\cite{BrzMa} and \cite[Chapter~5]{BegMa}. We also require $\omega$ to be such that $(\id-\Pi_\omega)\extd P\subseteq \Omega^1_AP$, in which case it defines a~connection on $P$ itself by
\[ \nabla_P\colon \ P\to \Omega^1_AP=\Omega^1_A\tens_A P,\qquad \nabla_P p= \extd p- p\rz\omega(\varpi\pi_\eps p\co).\]
The first equality is canonically afforded by viewing $\Omega^1_A\subseteq\Omega^1_P$ and multiplying there.
It is known\cite[Proposition~5.50]{BegMa} that this is a bimodule connection if and only if $\sum a_i\omega\big(\Lambda^1_H\big)b_i=0$ for all $\sum a_i \tens b_i\in A\tens A$ such that $\sum a_ib_i=0$ and $\sum (\extd a_i)b_i=0$. In this case, the generalised braiding is
\[ \sigma_P(p\tens (\extd a) b)= p(\extd a)b-p\rz \big[a,\omega\big(\varpi\pi_\eps p\co\big)\big]b.\]
Moreover, $\nabla_P$ commutes with the right coaction of $H$ under our assumptions. As a~consequence, for any $H$-comodule $V$, we can define $E=(P\tens V)^H$ as a noncommutative analogue of the space of sections of the associated bundle (the superscript~$H$ indicates the invariant subspace under the tensor product coaction) and a connection~\cite{BegMa, BrzMa}
\[ \nabla_E\colon \ (P\tens V)^H=E\to \big(\Omega^1_A P\tens V\big)^H=\Omega^1_A\tens_A E,\qquad \nabla_E=\nabla_P\tens \id.\]

We see that the key here for the theory of associated bundles and `spin connections' is the construction of an equivariant $\nabla_P$. Its properties of being a bimodule connection when viewing $A\subseteq P$ and working inside~$\Omega^1_P$ come down to
\begin{gather} \nabla_P\colon \ P\to \Omega^1_AP,\qquad \sigma_P\colon \ P\Omega^1_A\to \Omega^1_AP,\nonumber\\
 \label{nablaPleib1} \nabla_P(a.p) =(\extd a).p+ a.\nabla_P(p),\qquad \nabla_P(p.a)=(\nabla_P p).a+ \sigma_P(p.\extd a),\\ \label{nablaPleib2} a.\sigma_P(p.\extd b) =\sigma_P(a.p.\extd b),\qquad (\sigma_P(p.\extd b)).a=\sigma_P(p.(\extd b).a)\end{gather}
 for all $a,b\in A$, $p\in P$. We constructed $\nabla_P$ from $\omega$ above with conditions to ensure that~$\sigma_P$ is well-defined by (\ref{nablaPleib1}), after which it automatically obeys the $A$-bimodule map conditions~(\ref{nablaPleib2}). We emphasise the product of $P$ and $\Omega^1_P$ for clarity. There are exactly similar inherited formulae for~$\nabla_E$, now with $\sigma_E\colon \big(P\Omega^1_A\tens V\big)^H\to \big(\Omega^1_AP\tens V\big)^H$. The point is that by working `upstairs' in~$\Omega^1_P$, we avoid mention of~$\tens_A$.

What this comes down to for a classical principal bundle $P=C^\infty(X)$, Lie structure group $G$ and Lie algebra $\cg$ with basis $\{e_i\}$, say, (with $\Lambda^1_H$ replaced now by $\cg^*$) is that a spin connection is an equivariant Lie-algebra valued 1-form on $X$,
 \begin{equation}\label{classomega} \omega=\omega^i\tens e_i,\qquad \xi\la \omega^i+\omega^i\tens [\xi,e_i]=0,\qquad i_{\tilde\xi}(\omega^i) e_i=\xi\end{equation}
 for all $\xi\in\cg$ and sum over $i$ understood. Then
\begin{equation}\label{classnabla}\nabla_P\colon \ P\to \Omega^1_{\rm hor},\qquad \nabla_Pp= \extd p- \omega^i\tilde{e_i}(p),\qquad \sigma_P(p\extd a)=(\extd a)p\end{equation}
obeys the above Leibniz rules with everything commutative (so $\sigma_P=\id$ on $\Omega^1_{\rm hor}$). From this, restricting to $E=C^\infty_G(X,V)$ (the sections of the vector bundle associated to a~representation~$V$ of~$G$) gives the induced connection $\nabla_E$ on the associated bundle. Interior product with a vector field gives the covariant derivative along that vector field in the usual sense.

We now suppose that we have such a classical principal bundle with spin connection $\omega$ and seek to deformation-quantise everything. In fact, we are only going to look at this up to and including first order in $\lambda$, what could be called the `semiclassical' level. Formally, this can be done by working over the ring $\C[\lambda]/\big(\lambda^2\big)$ in place of $\C[[\lambda]]$ as explained in \cite{BegMa:prg}, but for continuity with the general scheme, we will use the deformation point of view but ignore $\mathcal{O}\big(\lambda^2\big)$ as errors to be corrected in a 2nd order theory. We can assume that $C^\infty(X)[[\lambda]]$ has an associative $*$-product but we only assume this $\mathcal{O}\big(\lambda^2\big)$ for the bimodule product $\bullet$ of $\Omega^1(X)[[\lambda]]$, since we allow that $\hat\nabla^X$ could have curvature. We do, however, fix the antisymmetric form of the leading products as per (\ref{proddef}) applied now on $X$ as a Poisson manifold. This means that our theory has a~slightly different flavour from previous sections, being specific to the leading part of a preferred deformation scheme.

\begin{Lemma}\label{nablabullet} Let $X$, $G$ be a Poisson principal bundle as in Definition~{\rm \ref{PPBdef}} and suppose we are given a connection in the sense
\[ \nabla_P\colon \ C^\infty(X)\to \Omega^1_{\rm hor},\qquad \nabla_P(ap)=(\extd a)p +a\nabla_Pp\]
for all $p\in C^\infty(X), a\in C^\infty(M)$. Consider $\Omega^1(X)[[\lambda]]$ quantised as in \eqref{proddef} and $\nabla^\bullet_Pp$ in here of the form
\[ \nabla^\bullet_{P}p= \nabla_Pp + \frac{\lambda}{2} \Gamma(p)+ \mathcal{O}\big(\lambda^2\big),\qquad \Gamma\colon \ C^\infty(X)\to \Omega^1_{\rm hor}.\]
Then the first of \eqref{nablaPleib1} holds with $*$ and $\bullet$ products to $\mathcal{O}\big(\lambda^2\big)$ if and only if
\begin{equation}\label{Gamma} \Gamma(ap)=a\Gamma(p)+\hat\nabla^X_{\extd a}\nabla_Pp-\hat\nabla^X_{\extd p}\extd a-\nabla_P\{a,p\}_X,\end{equation}
while the second of \eqref{nablaPleib1} holding requires
\begin{gather*} \sigma^\bullet_P(p\extd a) =p\extd a+\lambda\varsigma(p\extd a)+\mathcal{O}\big(\lambda^2\big),\qquad \varsigma\colon \ \Omega^1_{\rm hor}\to \Omega^1_{\rm hor},\nonumber\\
%\label{sigmaPsemi}
\varsigma(p\extd a) =\hat\nabla^X_{\extd a}\nabla_Pp-\hat\nabla^X_{\extd p}\extd a-\nabla_P\{a,p\}_X. \end{gather*}
If this is well-defined then \eqref{nablaPleib2} hold with $*$ and $\bullet$ products to $\mathcal{O}\big(\lambda^2\big)$. \end{Lemma}
\begin{proof} This is simply a matter of putting the form of the $*,\bullet$ products and $\nabla^\bullet_{P}$ into (\ref{nablaPleib1})--(\ref{nablaPleib2}) and discarding anything $\mathcal{O}\big(\lambda^2\big)$. For example,
\begin{gather*} \nabla^\bullet_{P}(a*p) =\nabla_P(ap)+{\lambda\over 2}\nabla_P\{a,p\}_X+{\lambda\over 2}\Gamma(ap)+ \mathcal{O}\big(\lambda^2\big),\\
 (\extd a)\bullet p+a\bullet\nabla^\bullet_{P}p =(\extd a)p-{\lambda\over 2}\hat\nabla^X_{\extd p}\extd a+a\nabla_P p+{\lambda\over 2}\hat\nabla^X_{\extd a}\nabla_Pp+a{\lambda\over 2}\Gamma(p)+ \mathcal{O}\big(\lambda^2\big)\end{gather*}
 gives the first result. For $\sigma_P^\bullet$ we first establish that $\sigma_P^\bullet(p\extd a)=p\extd a+O(\lambda)$ by computing the right hand side at classical order, then $\sigma(p\bullet\extd a)=\sigma_P^\bullet\big(p\extd a+{\lambda\over 2}\hat\nabla^X_{\extd p}\extd a\big)=\sigma_P^\bullet(p\extd a)+{\lambda\over 2}\hat\nabla^X_{\extd p}\extd a$ to $\mathcal{O}\big(\lambda^2\big)$ and we equate this to $\nabla^\bullet_{P}(p*a)-(\nabla^\bullet_{P})\bullet a$, from which $\Gamma$ conveniently cancels given our result for $\Gamma(ap)$. That (\ref{nablaPleib2}) holds to $\mathcal{O}\big(\lambda^2\big)$ is a tedious calculation. E.g., $a\bullet\sigma(p\extd b)=\sigma(a\bullet(p\extd b))$, comes down to $\hat\nabla^X_{\extd a}\extd b-\hat\nabla^X_{\extd b}\extd a=\extd\{a,b\}_X$ on expanding both sides and using the stated formula for $\sigma_P^\bullet$. Details are omitted. Finally, note that $\eta\in \Omega^1_{\rm hor}$ is characterised by $i_{\tilde\xi}\eta=0$ for all $\xi\in \cg$ and the transversality condition in Definition~\ref{PPBdef} then ensures that $\hat\nabla^X_\tau$ preserves $\Omega^1_{\rm hor}$ for all $\tau\in \Omega^1(X)$ (because $i_{\tilde\xi}\hat\nabla^X_\tau\eta=0$ for all $\eta\in \Omega^1_{\rm hor}$). It is then clear that the image of the stated formula for $\varsigma$ is in $\Omega^1_{\rm hor}$. \end{proof}

This lemma clarifies what we mean by quantizing a spin connection in the~$\nabla_P$ form to semiclassical order as a bimodule connection at this level. The order $\lambda$ data for the associated bundle connection is just $\Gamma$ restricted to $f\in E=C^\infty_G(V)$ with $\nabla_E^\bullet f= \nabla_E f+\frac{\lambda}{2} \Gamma(f)+\mathcal{O}\big(\lambda^2\big)$.

Our approach to construct $\Gamma$ and $\varsigma$ is to start with a deformed `spin connection' which, at order $\lambda$, can have an additional component
 \begin{equation}\label{classalpha} \omega_\bullet=\omega+ {\lambda\over 2}\alpha^i\tens e_i+ \mathcal{O}\big(\lambda^2\big),\qquad \xi\la\alpha^i\tens e_i+ \alpha^i\tens [\xi,e_i]=0,\end{equation}
where $\alpha^i\in \Omega^1_{\rm hor}$, and follow the lines of a deformed version of (\ref{classnabla}). We can take $\alpha=0$ as a~canonical choice here.

\begin{Theorem}\label{semiclassbundleconn} Let $X$, $G$ be a Poisson principal bundle as in Definition~{\rm \ref{PPBdef}}, $\omega$ a classical `spin connection' on~$X$ with corresponding $\nabla_P$ in \eqref{classnabla}, and $\alpha$ a horizontal equivariant $\cg$-valued $1$-form on $X$ as in \eqref{classalpha}. Then
\begin{equation}\label{Gammaomega}\Gamma(p)=\widetilde{\Xi_{e_i}^{*1}}\big(\widetilde{\Xi_{e_i}^{*2}}(p)\big)\omega^i- \hat\nabla^X_{\extd \tilde{e_i}(p)}\omega^i-\tilde{e_i}(p)\alpha^i,\qquad \varsigma(\tau)=- \hat\nabla^X_{e_i\la\tau}\omega^i\end{equation}
for all $p\in C^\infty(X)$ and $\tau\in \Omega^1_{\rm hor}$ constructs $\nabla^\bullet_{P} =\nabla_P +{\lambda\over 2}\Gamma+ \mathcal{O}\big(\lambda^2\big)$ and $\sigma^\bullet_P=\id+\lambda\varsigma+\mathcal{O}\big(\lambda^2\big)$ as a~bimodule connection at semiclassical order by Lemma~{\rm \ref{nablabullet}}. Moreover, $\Gamma$ and hence $\nabla^\bullet_P$ to~$\mathcal{O}\big(\lambda^2\big)$ are equivariant.
\end{Theorem}
\begin{proof} We look at $\Gamma$ first. Using the bundle transversality condition, we have
\[ i_{\tilde\xi}\nabla^X_{\extd \tilde{e_i}(p)}\omega^i=\widetilde{\Xi_\xi^{*1}}(\tilde{e_i}(p)) i_{\widetilde{\Xi_\xi^{*2}}}\big(\omega^i\big)+\big\{\tilde{e_i}(p),i_{\tilde\xi}\big(\omega^i\big)\big\}_X =\widetilde{\Xi_\xi^{*1}}\big({\widetilde{\Xi_\xi^{*2}}}(p)\big)=i_{\tilde\xi}\big(\widetilde{\Xi_{e_i}^{*1}} \big({\widetilde{\Xi_{e_i}^{*2}}}(p)\big)\omega^i\big),\]
where $i_{\tilde\xi}\big(\omega^i\big)=\xi^i$ for a spin connection as in~(\ref{classomega}), these being constants as $\xi=\xi^ie_i\in\cg$. Hence the first two terms of~$\Gamma$ are horizontal. The last term of~$\Gamma$ is horizontal as the $\alpha^i$ are assumed so. In fact, to be horizontal, we just need $\tilde{e_i}(p)i_{\tilde\xi}\big(\alpha^i\big)=0$ for all~$p$. But since the classical action is free, $\ver$ is surjective and it follows that we need $\alpha^i\in \Omega^1_{\rm hor}$ if we want the image of~$\Gamma$ to be horizontal.

Next, for $a\in C^\infty(M)$, we have $\tilde\xi(ap)=a\tilde\xi(p)$ for all $\xi\in \cg$ from which it is clear that
\[ \Gamma(ap)-a\Gamma(p)=-\hat\nabla^X_{\extd(a\tilde{e_i}(p))}\omega^i+a\hat\nabla^X_{\extd\tilde{e_i}(p)}\omega^i=-\tilde{e_i}(p)\hat\nabla^X_{\extd a}\omega^i,\]
while
\begin{gather*}\hat\nabla^X_{\extd a}\nabla_Pp-\hat\nabla^X_{\extd p}\extd a -\nabla_P\{a,p\}_X\\
\qquad{} =\hat\nabla^X_{\extd a}\extd p-\hat\nabla^X_{\extd a}(\tilde{e_i}(p)\omega^i)-\hat\nabla^X_{\extd p}\extd a-\extd\{a,p\}_X+e_i\la\{a,p\}_X\omega^i\\
\qquad{} =-\tilde{e_i}(p)\hat\nabla^X_{\extd a}\omega^i-\{a,\tilde{e_i}(p)\}\omega^i +e_i\la\{a,p\}_X\omega^i=-\tilde{e_i}(p)\hat\nabla^X_{\extd a}\omega^i\end{gather*}
also, by Lemma~\ref{eqPLGprop} and the Poisson-compatibility of $\hat\nabla^X$. This also gives us \[ \varsigma(p\extd a) =-\tilde{e_i}(p)\hat\nabla^X_{\extd a}\omega^i =-\hat\nabla^X_{\tilde{e_i}(p)\extd a}\omega^i =-\hat\nabla^X_{e_i\la(p\extd a)}\omega^i\]
since $e_i\la$ is the Lie derivative along $\tilde{e_i}$ and acts trivially on 1-forms on the base. Hence, this is well-defined on $\Omega^1_{\rm hor}$ and as stated. Since it is also $\Gamma(ap)-a\Gamma(p)$, it follows that it lies in $\Omega^1_{\rm hor}$ also, as one can also directly check using the bundle transversality condition.

So far, we have not used the assumed equivariance properties of $\omega$ and $\alpha$. We now assume these. Then
\begin{align*}\xi\la\Gamma(p)&=\big(\tilde\xi\widetilde{\Xi^{*1}_{e_i}}\widetilde{\Xi^{*2}_{e_i}}(p)\big)\omega^i + \big(\widetilde{\Xi^{*1}_{e_i}}\widetilde{\Xi^{*2}_{e_i}}(p)\big)(\xi\la\omega^i) -(\xi e_i\la p)\alpha^i- (e_i\la p)\big(\xi\la\alpha^i\big)- \xi\la \hat\nabla^X_{\extd e_i\la p}\omega^i\\
&=\big(\tilde\xi\widetilde{\Xi^{*1}_{e_i}}\widetilde{\Xi^{*2}_{e_i}}(p)\big)\omega^i
 - \big(\widetilde{\Xi^{*1}_{[\xi,e_i]}}\widetilde{\Xi^{*2}_{[\xi,e_i]}}(p)\big)\omega^i-(e_i\xi\la p)\alpha^i\\
 &\quad{} -\hat\nabla^X_{\extd \xi e_i\la p}\omega^i -\hat\nabla^X_{\extd e_i\la p}\big(\xi\la \omega^i\big)-\big( \delta^1_\xi e_i\la p\big)\big(\delta^2_\xi\la \omega^i\big)\\
 &= \big(\xi\Xi^{*1}_{e_i}\Xi^{*2}_{e_i}\la p\big)\omega^i
 - \big(\Xi^{*1}_{[\xi,e_i]}\Xi^{*2}_{[\xi,e_i]}\la p\big)\omega^i-(e_i\xi\la p)\alpha^i-\hat\nabla^X_{\extd e_i\xi\la p}\omega^i+\big( \delta^1_\xi [\delta^2_\xi,e_i]\la p\big)\omega^i \\
 &= \big(\Xi^{*1}_{e_i}\Xi^{*2}_{e_i}\xi\la p\big)\omega^i-(e_i\xi\la p)\alpha^i-\hat\nabla^X_{\extd e_i\xi\la p}\omega^i=\Gamma(\xi\la p)
\end{align*}
as required. We used equivariance of $\alpha$ to transfer $\xi\la\alpha^i$ to $[\xi,e_i]$, then covariance of $\hat\nabla^X$ in our Poisson sense of Definition~\ref{concovX} to expand $\xi\la\hat\nabla^X$ and equivariance of the classical $\omega^i$ as in (\ref{classomega}) to transfer any actions $\xi\la\omega^i$ and $\delta^2_\xi\la\omega^i$ to the relevant~$e_i$. We finally used the bicovariance condition in terms of~$\Xi$ in Theorem~\ref{bicovXi} to recognise the answer. We already have $\nabla_P$ equivariant, so $\nabla_P^\bullet$ would be quantum group covariant to $\mathcal{O}\big(\lambda^2\big)$ similarly to Remark~\ref{remomegaXcov}.
 \end{proof}

 \begin{Example} Our theme of the $q$-Hopf fibration gives a modest example. Indeed, the $q$-monopole connection is known at the quantum group level \cite{BegMa, BrzMa} and in the classical limit constructs the classical monopole by the choice $\omega \big(t^{-1}\extd t \big)=e^0$, or in our terms $\omega=e^0\tens H$. We have seen that $\Xi_H^*=-H\tens H$ in Example~\ref{exS2PB}, so
 \[ \Gamma(p)=\tilde H(\tilde H(p)) e^0-\hat\nabla^X_{\extd\tilde H(p)}e^0-\tilde H(p)\alpha=-|p|p\alpha\]
 for any $\alpha\in \Omega^1_{\rm hor}$ and $p$ of homogenous degree. Here the first two terms cancel using our results in Example~\ref{exSU2}. Representations of $S^1$ are labelled by $n\in \Z$ and the associated `charge~$n$ monopole' bundle has sections $E_n=C^\infty_{S^1}({\rm SU}_2)$ with dense subspace given by restricting to elements $p$ of homogeneous degree~$-n$. We see that the correction to the classical monopole $\nabla_{E_n}$ at first order is just to add~$n p\alpha$.
 \end{Example}

Finally, we outline how the formula (\ref{Gamma}) could be obtained from a deformation-quantisation. This should be seen purely as motivation.

\begin{Remark}\label{wedgeHequalsg} We return to the setting of Remark~\ref{defbunrem}, where we suppose for the sake of discussion that $*$-product quantisations $C^\infty(X)[[\lambda]]$, $C^\infty(M)[[\lambda]]$ and $C^\infty(G)[[\lambda]]$ form a quantum principal bundle when working over $\C[[\lambda]]$, at least to $\mathcal{O}\big(\lambda^2\big)$. We assume as there that the coproduct and $\extd$ are undeformed and note that the deformed antipode is $S^\bullet h=S h -{\lambda\over 2}\{S h\o,h\t\}S h\th+\mathcal{O}\big(\lambda^2\big)$. Then
\begin{align*} \varpi^\bullet (v)&:=S^\bullet v\o\bullet \extd v\t= S v\o\bullet \extd v\t- {\lambda\over 2}\{Sv\o,v\t\}(Sv\th)\extd v\fo\\
&=\varpi(v)+{\lambda\over 2}\big(\hat\nabla^G_{\extd Sv\o}\extd v\t-\{Sv\o,v\t\}\varpi\pi_\eps v\th\big)\\
&=\varpi(v)+ {\lambda\over 2}v\t \hat\nabla^G_{\extd Sv\o}\varpi\pi_\eps v\th=\varpi(v)-{\lambda\over 2}\hat\nabla^G_{\varpi\pi_\eps v\o}\varpi\pi_\eps v\t\end{align*}
to errors $\mathcal{O}\big(\lambda^2\big)$, where $v$ is a function on $G$ vanishing at $e$, viewed in $H^+$. For the third equality, we used $\extd v\t= v\t \varpi\pi_\eps v\th$, that $\hat\nabla^G$ is a contravariant connection, the Leibniz rule and $S^2=\id$ on the classical group to identify its subscript. We now view this as a map
\begin{equation}\label{mcbullet} \varpi^\bullet \colon \ H^+\to \cg^*[[\lambda]],\qquad \<\xi,\varpi^\bullet (v)\>=\tilde\xi(v)(e) - {\lambda\over 2}\widetilde{\Xi_\xi^{*1}}(\widetilde{\Xi_\xi^{*2}}(v))(e)+\mathcal{O}\big(\lambda^2\big),\end{equation}
 where $e$ is the group identity, on noting that
 \begin{equation*}\label{Xi2} \<\xi,\varpi\pi_\eps(v\o)\>\<\eta,\varpi\pi_\eps(v\t)\>={\extd\over\extd t}\Big|_0{\extd \over\extd s}\Big|_0v\big(e^{t\xi}e^{s\eta}\big)=\tilde\xi(\tilde\eta(v))(e)\end{equation*}
and $\<\xi,\varpi(v)\>={\extd \over\extd t}\big|_0v\big(e^{t\xi}\big)=\tilde\xi(v)(e)$ for any functions $v$ and $\xi,\eta\in\cg$ (we apply the first observation to the output of $\Xi^*\in \cg\tens\cg$). In the quantum differential calculus, we identify left-invariant 1-forms $\Lambda^1_H$ with $H^+$ modulo the kernel of $\varpi^\bullet $ and on the quantum bundle we suppose a spin connection $\omega_\bullet\colon \Lambda^1_H\to \Omega^1(X)[[\lambda]]$ which we take as given by~(\ref{mcbullet}) followed by~(\ref{classalpha}) viewed as a~map from $\cg^*$, to give
\[ \nabla^\bullet_P p=\extd p - p\rz \bullet \omega^i\<e_i,\varpi^\bullet (p\co)\>=\nabla_P p +{\lambda\over 2}\Gamma(p)+\mathcal{O}\big(\lambda^2\big)\]
for $\Gamma(p)$ as in (\ref{Gammaomega}), after a calculation to expand the bullets to semiclassical order. \end{Remark}

\section{Concluding remarks}

We have shown that Drinfeld's theory of Poisson--Lie groups~\cite{Dri:ham, Drinfeld} extends in a natural way to principal bundles $X$ which are Poisson manifolds and have Poisson--Lie structure group~$G$. The PLG condition itself is viewed as a certain covariance of the Poisson bracket $\{\ ,\ \}_G$ and extends to the notion of a PLG action on $\{\ ,\ \}_X$. We extended these covariance notions to Poisson-level quantum differential structures in the sense of Poisson-compatible contravariant connections~$\hat\nabla^G$ and~$\hat\nabla^X$ respectively in the sense of \cite{BegMa:sem,Fer,Haw, Hue,MaTao:plg} and characterised bicovariance of the former. Section~\ref{PPB} introduced a further transversality condition expressing that $X\to M=X/G$ is a Poisson-level version of a quantum principal bundle. This is such that $M$ is not only a Poisson manifold but inherits a Poisson-level quantum differential structure in the sense of $\hat\nabla^M$. We were then able to formulate `spin connections' and connections on associated bundles at semiclassical level. Results were illustrated on the Poisson level of the $q$-Hopf fibration deforming ${\rm SU}_2\to S^2$ with $S^1$ fibre, the latter as a PLG with zero $\{\ ,\ \}_G$ but nonzero $\hat\nabla^G$. It remains to compute more complicated examples.

A direction for further work would be a general Poisson version of the theory of quantum homogeneous bundles, of which the $q$-Hopf fibration above is the simplest example. This theory with semiclassical differential structures would be rather different from previous discussions of Poisson homogeneous spaces, for example in~\cite{Dri:hom}. Whether our approach can extend to a~geometric picture of the full 2-parameter Podle\'s spheres~\cite{Pod} is unclear since a direct quantum principal bundle approach would seem to require coalgebra bundles~\cite{BrzMa:coa}, for which the notion of differential calculus on the fibre is not known. Another issue, even for ordinary quantum principal bundles, is that associative bicovariant calculi on $q$-deformation quantum groups are often not possible with classical dimension~\cite{BegMa:sem}. This was not a problem in our approach at the Poisson level as we were careful not to assume flatness of $\hat\nabla^G$, but quantisation would be a~challenge. It should be possible, for example, to construct a nonassociative deformation of the next Hopf fibration $S^7 \to S^4$ with ${\rm SU}_2$ fibre, combining the Hopf-quasigroup methods in~\cite{KliMa} and the nonassociative bicovariant calculi by twisting in~\cite{BegMa:sem,BegMa:twi}.

\looseness=-1 A third interesting direction would be quantum group frame bundles as in~\cite{Ma:non} at the Poisson level whereby $\Omega^1(M)$ is an associated bundle to a Poisson principal bundle. In this context, it could make sense to ask that $\hat\nabla$ comes from a contravariant connection on the underlying classical principal bundle in the sense of~\cite{Fer}. The framed theory would, moreover, extend the above to a PLG-covariant version of Riemannian geometry on $M$ as a somewhat different approach to~\cite{BegMa:prg}. Indeed, one expects the quantum metrics at semiclassical order to be different in the two approaches even for the Poisson level of the metric for the $q$-sphere. It could also be of interest to consider quantum fibrations more generally, and to relate to integrability of Lie algebroids~\cite{CraFer}.

\vspace{-2mm}

\pdfbookmark[1]{References}{ref}
\LastPageEnding


\begin{thebibliography}{99}
\footnotesize\itemsep=-1pt

\bibitem{BegMa:sem}
Beggs E.J., Majid S., Semiclassical differential structures, \href{https://doi.org/10.2140/pjm.2006.224.1}{\textit{Pacific J.
 Math.}} \textbf{224} (2006), 1--44, \href{https://arxiv.org/abs/math.QA/0306273}{arXiv:math.QA/0306273}.

\bibitem{BegMa:twi}
Beggs E.J., Majid S., Quantization by cochain twists and nonassociative
 differentials, \href{https://doi.org/10.1063/1.3371677}{\textit{J.~Math. Phys.}} \textbf{51} (2010), 053522, 32~pages,
 \href{https://arxiv.org/abs/math.QA/0506450}{arXiv:math.QA/0506450}.

\bibitem{BegMa:prg}
Beggs E.J., Majid S., Poisson--{R}iemannian geometry, \href{https://doi.org/10.1016/j.geomphys.2016.12.012}{\textit{J.~Geom. Phys.}}
 \textbf{114} (2017), 450--491, \href{https://arxiv.org/abs/1403.4231}{arXiv:1403.4231}.

\bibitem{BegMa}
Beggs E.J., Majid S., Quantum {R}iemannian geometry, \textit{Grundlehren der
 mathematischen Wissenschaften}, Vol.~355, \href{https://doi.org/10.1007/978-3-030-30294-8}{Springer}, Berlin, 2020.

\bibitem{Bon}
Bonneau P., Flato M., Gerstenhaber M., Pinczon G., The hidden group structure
 of quantum groups: strong duality, rigidity and preferred deformations,
 \href{https://doi.org/10.1007/BF02099415}{\textit{Comm. Math. Phys.}} \textbf{161} (1994), 125--156.

\bibitem{Bor}
Bordemann M., Neumaier N., Waldmann S., Wei{\ss} S., Deformation quantization
 of surjective submersions and principal fibre bundles, \href{https://doi.org/10.1515/CRELLE.2010.009}{\textit{J.~Reine
 Angew. Math.}} \textbf{639} (2010), 1--38, \href{https://arxiv.org/abs/0711.2965}{arXiv:0711.2965}.

\bibitem{BrzMa}
Brzezi\'nski T., Majid S., Quantum group gauge theory on quantum spaces,
 \href{https://doi.org/10.1007/BF02096884}{\textit{Comm. Math. Phys.}} \textbf{157} (1993), 591--638,
 \href{https://arxiv.org/abs/hep-th/9208007}{arXiv:hep-th/9208007}.

\bibitem{BrzMa:coa}
Brzezi\'nski T., Majid S., Quantum geometry of algebra factorisations and
 coalgebra bundles, \href{https://doi.org/10.1007/PL00005530}{\textit{Comm. Math. Phys.}} \textbf{213} (2000), 491--521,
 \href{https://arxiv.org/abs/math.QA/9808067}{arXiv:math.QA/9808067}.

\bibitem{Bur}
Bursztyn H., Poisson vector bundles, contravariant connections and
 deformations, \href{https://doi.org/10.1143/PTPS.144.26}{\textit{Progr. Theoret. Phys. Suppl.}} \textbf{144} (2001),
 26--37.

\bibitem{Con}
Connes A., Noncommutative geometry, Academic Press, Inc., San Diego, CA, 1994.

\bibitem{CraFer}
Crainic M., Fernandes R.L., Integrability of {L}ie brackets, \href{https://doi.org/10.4007/annals.2003.157.575}{\textit{Ann. of
 Math.}} \textbf{157} (2003), 575--620, \href{https://arxiv.org/abs/math.DG/0105033}{arXiv:math.DG/0105033}.

\bibitem{DS}
Dazord P., Sondaz D., Groupes de {P}oisson affines, in Symplectic Geometry,
 Groupoids, and Integrable Systems ({B}erkeley, {CA}, 1989), \textit{Math.
 Sci. Res. Inst. Publ.}, Vol.~20, \href{https://doi.org/10.1007/978-1-4613-9719-9_6}{Springer}, New York, 1991, 99--128.

\bibitem{Dri:ham}
Drinfeld V.G., Hamiltonian structures on {L}ie groups, {L}ie bialgebras and the
 geometric meaning of the classical {Y}ang--{B}axter equations, \textit{Sov.
 Math. Dokl.} \textbf{27} (1983), 68--71.

\bibitem{Drinfeld}
Drinfeld V.G., Quantum groups, in Proceedings of the {I}nternational {C}ongress
 of {M}athematicians, {V}ols.~1,~2 ({B}erkeley, {C}alif., 1986), Amer. Math.
 Soc., Providence, RI, 1987, 798--820.

\bibitem{Dri:hom}
Drinfeld V.G., On {P}oisson homogeneous spaces of {P}oisson--{L}ie groups,
 \href{https://doi.org/10.1007/BF01017137}{\textit{Theoret. and Math. Phys.}} \textbf{95} (1993), 226--227.

\bibitem{DVM}
Dubois-Violette M., Michor P.W., Connections on central bimodules in
 noncommutative differential geo\-met\-ry, \href{https://doi.org/10.1016/0393-0440(95)00057-7}{\textit{J.~Geom. Phys.}} \textbf{20}
 (1996), 218--232, \href{https://arxiv.org/abs/q-alg/9503020}{arXiv:q-alg/9503020}.

\bibitem{Fer}
Fernandes R.L., Connections in {P}oisson geometry. {I}. {H}olonomy and
 invariants, \href{https://doi.org/10.4310/jdg/1214341648}{\textit{J.~Differential Geom.}} \textbf{54} (2000), 303--365,
 \href{https://arxiv.org/abs/math.DG/0001129}{arXiv:math.DG/0001129}.

\bibitem{FriMa}
Fritz C., Majid S., Noncommutative spherically symmetric spacetimes at
 semiclassical order, \href{https://doi.org/10.1088/1361-6382/aa72a5}{\textit{Classical Quantum Gravity}} \textbf{34} (2017),
 135013, 50~pages, \href{https://arxiv.org/abs/1611.04971}{arXiv:1611.04971}.

\bibitem{HajMa}
Hajac P.M., Majid S., Projective module description of the {$q$}-monopole,
 \href{https://doi.org/10.1007/s002200050704}{\textit{Comm. Math. Phys.}} \textbf{206} (1999), 247--264,
 \href{https://arxiv.org/abs/math.QA/9803003}{arXiv:math.QA/9803003}.

\bibitem{Haw}
Hawkins E., Noncommutative rigidity, \href{https://doi.org/10.1007/s00220-004-1036-4}{\textit{Comm. Math. Phys.}} \textbf{246}
 (2004), 211--235, \href{https://arxiv.org/abs/math.QA/0211203}{arXiv:math.QA/0211203}.

\bibitem{Hue}
Huebschmann J., Poisson cohomology and quantization, \href{https://doi.org/10.1515/crll.1990.408.57}{\textit{J.~Reine Angew.
 Math.}} \textbf{408} (1990), 57--113.

\bibitem{KliMa}
Klim J., Majid S., Hopf quasigroups and the algebraic 7-sphere,
 \href{https://doi.org/10.1016/j.jalgebra.2010.03.011}{\textit{J.~Algebra}} \textbf{323} (2010), 3067--3110, \href{https://arxiv.org/abs/0906.5026}{arXiv:0906.5026}.

\bibitem{KS}
Kosmann-Schwarzbach Y., Lie bialgebras, {P}oisson {L}ie groups and dressing
 transformations, in Integrability of Nonlinear Systems ({P}ondicherry, 1996),
 \textit{Lecture Notes in Phys.}, Vol.~495, \href{https://doi.org/10.1007/BFb0113695}{Springer}, Berlin, 1997, 104--170.

\bibitem{Lu}
Lu J.-H., Multiplicative and affine {P}oisson structures on {L}ie groups, Ph.D.~Thesis, {U}niversity of California, Berkeley, 1990.

\bibitem{Ma:book}
Majid S., Foundations of quantum group theory, \href{https://doi.org/10.1017/CBO9780511613104}{Cambridge University Press},
 Cambridge, 1995.

\bibitem{Ma:non}
Majid S., Riemannian geometry of quantum groups and finite groups with
 nonuniversal differentials, \href{https://doi.org/10.1007/s002201000564}{\textit{Comm. Math. Phys.}} \textbf{225} (2002),
 131--170, \href{https://arxiv.org/abs/math.QA/0006150}{arXiv:math.QA/0006150}.

\bibitem{Ma:spo}
Majid S., Noncommutative model with spontaneous time generation and {P}lanckian
 bound, \href{https://doi.org/10.1063/1.2084748}{\textit{J.~Math. Phys.}} \textbf{46} (2005), 103520, 18~pages,
 \href{https://arxiv.org/abs/hep-th/0507271}{arXiv:hep-th/0507271}.

\bibitem{Ma:spin}
Majid S., Noncommutative {R}iemannian and spin geometry of the standard
 {$q$}-sphere, \href{https://doi.org/10.1007/s00220-005-1295-8}{\textit{Comm. Math. Phys.}} \textbf{256} (2005), 255--285,
 \href{https://arxiv.org/abs/math.QA/0307351}{arXiv:math.QA/0307351}.

\bibitem{Ma:rec}
Majid S., Reconstruction and quantization of {R}iemannian structures,
 \href{https://doi.org/10.1063/1.5123258}{\textit{J.~Math. Phys.}} \textbf{61} (2020), 022501, 32~pages,
 \href{https://arxiv.org/abs/1307.2778}{arXiv:1307.2778}.

\bibitem{MaTao:plg}
Majid S., Tao W.-Q., Noncommutative differentials on {P}oisson--{L}ie groups and
 pre-{L}ie algebras, \href{https://doi.org/10.2140/pjm.2016.284.213}{\textit{Pacific~J. Math.}} \textbf{284} (2016), 213--256,
 \href{https://arxiv.org/abs/1412.2284}{arXiv:1412.2284}.

\bibitem{Mou}
Mourad J., Linear connections in non-commutative geometry, \href{https://doi.org/10.1088/0264-9381/12/4/007}{\textit{Classical
 Quantum Gravity}} \textbf{12} (1995), 965--974, \href{https://arxiv.org/abs/hep-th/9410201}{arXiv:hep-th/9410201}.

\bibitem{Pod}
Podle\'s P., Quantum spheres, \href{https://doi.org/10.1007/BF00416848}{\textit{Lett. Math. Phys.}} \textbf{14} (1987),
 193--202.

\bibitem{FRT}
Reshetikhin N.Yu., Takhtadzhyan L.A., Faddeev L.D., Quantization of {L}ie groups
 and {L}ie algebras, \textit{Leningrad Math.~J.} \textbf{1} (1990), 193--225.

\bibitem{SvdB}
Stafford J.T., van~den Bergh M., Noncommutative curves and noncommutative
 surfaces, \href{https://doi.org/10.1090/S0273-0979-01-00894-1}{\textit{Bull. Amer. Math. Soc. (N.S.)}} \textbf{38} (2001),
 171--216, \href{https://arxiv.org/abs/math.RA/9910082}{arXiv:math.RA/9910082}.

\bibitem{Vai}
Vaisman I., Lectures on the geometry of {P}oisson manifolds, \textit{Progress
 in Mathematics}, Vol.~118, \href{https://doi.org/10.1007/978-3-0348-8495-2}{Birkh\"auser Verlag}, Basel, 1994.

\bibitem{Wor}
Woronowicz S.L., Differential calculus on compact matrix pseudogroups (quantum
 groups), \href{https://doi.org/10.1007/BF01221411}{\textit{Comm. Math. Phys.}} \textbf{122} (1989), 125--170.

\end{thebibliography}
\end{document}